\def\seq#1_#2{\langle #1_#2:#2\in\omega\rangle}
\def\fc#1|#2{#1\uparrow#2}
\def\ain{\subseteq^*}
\def\cl#1{\overline{#1}}

\def\K{{\cal K}}

\def\C{{\mathfrak c}}
\def\CC{{\cal C}}

\def\set#1:#2.{{\{\,#1: #2\,\}}}

\def\so{\mathop{\mathfrak{so}}}

\def\F{{\mathbb F}}
\newif\ifdraft
\drafttrue

\documentclass{article}
\title{Countably compact groups and sequential order}
\author{Dmitri Shakhmatov, Alexander Shibakov}
\usepackage{amssymb}
\usepackage{amsthm}
\usepackage{amsmath}
\ifdraft\usepackage{showlabels}\fi
\usepackage{enumitem}
\usepackage{scrextend}
\newtheorem{theorem}{Theorem}
\newtheorem{lemma}{Lemma}

\newtheorem{corollary}{Corollary}

\newtheorem{definition}{Definition}
\newtheorem{question}{Question}
\newlist{countup}{enumerate}{10}
\setlist[countup]{label={\rm(\arabic*)}, ref=(\arabic*)}

\begin{document}
\maketitle
\begin{abstract}
We use $\diamondsuit$ to construct, for every $\alpha\leq\omega_1$ a
sequential countably compact topological group of sequential order
$\alpha$. This establishes the independence of the existence of
sequential countably compact non Fr\'echet
groups from the usual axioms of ZFC and answers several questions of
D.~Shakhmatov.
\end{abstract}
\def\lad{\la{D}}
\def\las{\la{S}}
\def\lasm{\la{S^m}}
\def\lasi{\la{S^i}}
\def\lade{\la{2D}}
\def\lase{\la{2S}}
\def\lasem{\la{2S^m}}
\def\lasei{\la{2S^i}}
\def\lado{\lad\setminus\lade}
\def\laso{\las\setminus\lase}
\def\lasom{\lasm\setminus\lasem}
\def\lasoi{\lasi\setminus\lasei}
\def\laf{\la{F}}
\def\U{{\cal U}}
\def\kw{\mathop{k_\omega}}
\def\Pty{{\bf p}}
\def\B{{\mathbb B}}
\def\UU{{\mathbb U}}
\def\S{{\mathbb S}}
\def\SS{{\cal S}}
\def\limD{*}
\def\oG{\bar G}
\def\leql{\leq_L}
\def\geql{\not\leql}
\def\leqp{\leq_p}
\def\geqp{\geq_p}
\def\hgt{{\rm ht}}
\def\MM#1{${\mathbb M}(#1)$}
\def\kwk{{\kw(\K)}}
\def\la#1{{\langle{#1}\rangle}}
\def\le#1{\la{2{#1}}}
\def\lo#1{{\la{#1}\setminus\le{#1}}}
\def\hd#1#2#3{{\bf h}_{#3}^{#1}({#2})}
\def\hda#1#2#3{\hd{#1}{#2}{\la{#3}}}
\def\hdo#1#2#3{\hd{#1}{#2}{\lo{#3}}}
\def\ha#1#2#3{{^\omega\hd{#1}{#2}{#3}}}
\def\gku{{(G, \K, \U)}}
\def\gks#1{{(G#1, \K#1, \U#1)}}
\def\gkup{{(G', \K', \U')}}
\def\hku#1{{(H#1, \K#1, \U#1)}}
\def\Kw{\K_{\omega_1}}
\def\Uw{\U_{\omega_1}}
\def\Gw{G_{\omega_1}}
\def\mcofl{{\bf m}}
\def\ccl#1{{\bf c}{#1}}
\section{Introduction and notation.}
The standard definition of a topological space contains no reference
to convergence, sequential or otherwise. In common mathematical practice,
however, convergence often appears as a valuable tool in many
application areas for topology, such as analysis. In an effort to
formalize the relationship between the sequential convergence and
topology, the class of {\em Fr\'echet spaces\/} was defined in which the
closure operator is directly described in terms of limits of
convergent sequences. Later, a more general class of {\em sequential
spaces\/} was introduced in~\cite{Fr1} to encompass all the topological
spaces in which convergent sequences fully describe the topology (see
below for the definitions of these and other concepts used in the
introduction).

The addition of an algebraic structure appropriately coupled with the
topology (by requiring the operations to be continuous)
imposes a number of restrictions on both the topology and the
convergence. Thus separation axioms $T_1$--$T_{3{1\over2}}$ are
equivalent in topological groups, as are metrizability and the first
countability axiom (see~\cite{ArTk}). The investigation into
sequential topological groups was started by V.~Malykhin, P.~Nyikos,
I.~Protasov,
E.~Zelenyuk, and others (see~\cite{PZ1}, \cite{Ny1} and the
references therein) in
the 1970s and 1980s. In the following decade, a number of questions about
sequential topological groups (more generally, about `sequential
phenomena' in groups) had been stated, that have guided the future
development of this field.

The survey paper~\cite{Sha1} (see also~\cite{Sha2}) presents a
fairly complete
overview of the state of the art in the study of convergence in the
presence of an algebraic structure.
%It mentions, of course, the
%well-known Malykhin's problem about the existence of a non metrizable
%countable Fr\'echet topological group that was recently beautifully
%resolved in~\cite{HRG}, as well as a number of other natural questions
%about similar classes of spaces.
One of the problems posed in~\cite{Sha1} (Question~7.5) is the existence of a
countably compact sequential group that is not Fr\'echet. This
question is stated for precompact and
pseudocompact sequential groups, as well. It was shown in~\cite{Shi1}
that it is consistent with
the usual axioms of ZFC that there are no such countably compact
groups. The goal of this paper is to establish the consistency of the
existence of such groups thus showing that such existence is
independent of ZFC. Note that every {\em compact\/} sequential (or
even countably tight) topological group is metrizable (see~\cite{ArTk}).

Another question in~\cite{Sha1} deals with the measure of (sequential)
complexity of the closure operator in sequential groups, called {\em
the sequential order\/} (see below for a definition). Namely,
Question~7.4 asks if countably compact (pseudocompact, precompact)
sequential groups of arbitrary sequential orders exist, consistently, or otherwise.

A similar question was asked by P.~Nykos in~\cite{Ny1} about the class
of all sequential groups. The existence of sequential groups with
nontrivial ($\not\in\{0,1,\omega_1\}$) sequential orders was shown to
be independent of ZFC in~\cite{Shi1} and~\cite{Shi7}. The example in
this paper (see Theorem~\ref{cc.g}) thus establishes that the answer
to Question~7.4(iii) from~\cite{Sha1} (existence of countably compact
groups of arbitrary sequential orders)
is likewise independent of ZFC (for the classes of pseudocompact and
precompact groups it is still an open question whether it is consistent
that there are no such groups of sequential order $\omega_1$).

We use the standard set-theoretic notation, see~\cite{Ku}
and~\cite{ArTk}. All spaces are assumed to be regular unless stated
otherwise.

Abusing the notation somewhat we write $\sigma-1=\sigma'$ where
$\sigma\geq1$ is a successor ordinal and $\sigma'+1=\sigma$ or $\sigma=0$
and $\sigma'=-1$ (we do not go as far as call $-1$ an ordinal though). For brevity,
we use the term {\em increasing\/} to mean non decreasing, and use
{\em strictly increasing\/} when a stronger condition is assumed.
Define $\alpha\leql\beta$ as $\alpha<\beta$ for a limit $\beta$ and
$\alpha\leq\beta$ otherwise. Note that if $\alpha$ is a successor
$\alpha\leq\beta$ is equivalent to $\alpha\leql\beta$.

Our notation for various ordinal invariants is more detailed than customary
as we must frequently keep track of several topologies on the same
space. Whenever the topology is clear from the context we omit it from
the notation as well. The following
definition is the starting point for most arguments about convergence.
\begin{definition}\label{seq.cl}
Let $(X,\tau)$ be a topological space, $A\subseteq X$.
Define {\em the sequential closure of $A$},
$[A]^\tau_0=[A]^\tau=\set x\in X:\exists S\subseteq A, S\to x.$.
Given an ordinal $\sigma>0$, define $[A]_\sigma^\tau=[[A]_{\sigma-1}^\tau]^\tau$ if $\sigma$
is a successor and $[A]_\sigma^\tau=\cup_{\sigma'<\sigma}[A]_{\sigma'}^\tau$ otherwise.
\end{definition}

{\em Sequential topological spaces\/} may be defined as exactly those $(X,\tau)$ in which for every
$A\subseteq X$ there exists a $\sigma$ such that $\cl A^\tau=[A]_\sigma^\tau$. It is a quick argument
to show that in such spaces $\sigma\leq\omega_1$. 
\begin{definition}\label{pt.so}
Let $(X,\tau)$ be a sequential space, $A\subseteq X$, and $x\in X$.
Define $\so(x, A, \tau)=\inf\set\sigma:x\in[A]_\sigma^\tau.\cup\{\omega_1\}$.
\end{definition}
As a quick observation, if $x\in\cl{A}^\tau$ then
$\so(x, A, \tau)=\sigma<\omega_1$ is a successor
ordinal, and whenever $\sigma>0$ there are $x_n\in A$ such
that $x_n\to x$ in $\tau$ and $\so(x_n, A, \tau)=\sigma_n$ is an
increasing (non decreasing) sequence of ordinals such that
$\sigma_n<\sigma$ and $\sigma_n\to\sigma-1$.

The central ordinal invariant in the study of convergence can now be defined as
follows.
\begin{definition}\label{pt.so}
Let $(X,\tau)$ be a sequential space. Define {\em the sequential order of $X$},
$\so(X, \tau)=\sup\set\so(x, A, \tau):A\subseteq X, x\in\cl{A}^\tau.$.
\end{definition}

The construction below depends heavily on the algebraic properties of
the underlying group.
Recall that a group $G$ is called {\it boolean\/} if $a+a=0_G$ for any $a\in G$.
All such groups are abelian and may be viewed as vector spaces over $\F_2$. One can thus
consider (linearly) independent subsets of $G$ in the usual sense. If $A\subseteq G$
by $\la{A}$ we denote the span of $A$ in $G$. A convenient property of
boolean groups that is used without mentioning below is that $a+b=a-b$
for any $a,b\in G$ for a boolean $G$.

The following construction will be used often. 
\begin{definition}\label{lase}
Let $G$ be a boolean group, $S\subseteq G$ be an independent
subset. Define the {\em even subgroup of $\las$ (relative to $S$)\/} as $\lase=\la{S+S}$.
\end{definition}
Under most circumstances, the set $S$ will be clear from the context and the `relative to $S$'
part will be omitted. We assume below that all groups are abelian
unless stated otherwise, although a number of statements hold under
more general assumptions.

\begin{lemma}\label{set.additivity}
  Let $G$ be a topological group, $A, B\subseteq G$. Let
  $a\in[A]_\alpha$ and $b\in[B]_\beta$ for some
  $\alpha,\beta<\omega_1$. Then $a+b\in[A+B]_{\max\{\alpha,\beta\}}$.
\end{lemma}
\begin{proof}
We may assume that both $\alpha$ and $\beta$ are successor ordinals.
If $\max\{\alpha,\beta\}=0$ then $a\in A$ and $b\in B$ so
$a+b\in A+B=[A+B]_0$.

Suppose the Lemma holds for all successor $\alpha', \beta'$ such that
$\max\{\alpha',\beta'\}<\gamma$. Let $\so(a,A)=\alpha$ and
$\so(b,B)=\beta$, $\max\{\alpha,\beta\}=\gamma$. Pick  $b_n\to b$ and
$a_n\to a$ such that $\so(a_n, A)=\alpha_n$, $\so(b_n,B)=\beta_n$
where $\alpha_n<\gamma$, $\beta_n<\gamma$ are such that
$\alpha_n\to\alpha-1$ and $\beta_n\to\beta-1$. By the hypothesis
$a_n+b_n\in[A+B]_{\max\{\alpha_n,\beta_n\}}=[A+B]_{\gamma_n}$ where
$\gamma_n<\gamma$. Thus $a+b\in[A+B]_\gamma$.
\end{proof}

\begin{lemma}\label{cs.sos}
Let $G$ be a topological group and $K\subseteq G$ be a sequentially
compact subspace. Let $P\subseteq D+K$. Then for any
$\alpha<\omega_1$ $[P]_\alpha\subseteq[D]_\alpha+K$.
\end{lemma}
\begin{proof}
Suppose the lemma holds for all $\alpha'<\alpha$ and let $\so(x,
P)=\alpha$. If $\alpha=0$ the argument is trivial, otherwise, there
are $x_n=d_n+g_n$ such that $x_n\to x$, $d_n\in D$, $g_n\in K$ and
$\so(x_n, P)=\alpha_n<\alpha$. By the hypothesis, we may assume that
$x_n=d_n'+g_n'$ where $\so(d_n',D)\leq\alpha_n$ and $g_n'\in K$. Using
the sequential compactness of $K$ we may assume (after passing to a
subsequence if necessary) that $g_n'\to g$ for some $g\in K$. Thus
$d_n'\to d=x-g$ and $\so(d, D)=\alpha'\leq\alpha$.
\end{proof}

\begin{corollary}\label{cs.mseq}
Let $G$ be a topological group and $K\subseteq G$ be a sequentially
compact subspace. Let $P=\set p_n:n\in\omega.$ and
$p_n=a_n+d^n+d_n$ where $a_n\in K$,
$d_n\to d$, and $g\in[\set d^i:i\in\omega.]_\sigma$.

Then there exists a $p\in[P]_{\max\{\sigma,1\}}$ such that $p\in
g+d+K$. If $a_n\to a$ one may assume that
$p=a+g+d$.
\end{corollary}
\begin{proof}
Put $D=\set d^n+d_n:n\in\omega.$ and apply Lemma~\ref{set.additivity} to show that
$g+d\in[D]_{\max\{\sigma,1\}}$. Now note that $D\subseteq
P+(-K)$ and use Lemma~\ref{cs.sos} (with the roles of $P$ and $D$
reversed) to find a $p\in[P]_{\max\{\sigma,1\}}$
and an $a\in K$ such that $p=g+d+a$.

If $a_n\to a$ note that
$p=a+g+d\in[P]_{\max\{\sigma,1\}}$ by
Lemma~\ref{set.additivity}.
\end{proof}
The next concept is used to build approximations of the sequential
group topology.
\begin{definition}\label{kw.topology}
Let $(X,\tau)$ be a topological space. Then $(X,\tau)$ is called $\kw$
if there exists a countable family $\K$ of subspaces of $X$ such that
$U\in\tau$ if and only if $U\cap K$ is relatively open in $K$ for
every $K\in\K$. We say that $\tau$ is {\em determined} by $\K$ and
write $\tau=\kw(\K)$.
\end{definition}

A rich source of $\kw$ topologies on $X$ is provided by the following
well known construction. Let $\K$ be a countable family of subsets of $X$ such
that each $K\in\K$ is endowed with a compact topology $\tau_K$. Define
a new topology $\tau=\kw(\K)=\set U:\forall K\in\K\; (U\cap K)\in\tau_K.$. It is easy to see
that such $\tau$ is automatically $\kw$. The construction above makes
sense for uncountable $\K$ as well. We will use $\kw(\K)$ to denote
the appropriate topology even though for an uncountable $\K$, the
topology $\kw(\K)$ is not necessarily $\kw$.

Note that for an arbitrary $\K$ such $\tau$ is not guaranteed to be
Hausdorff (although it is always $T_1$ provided each $\tau_K$ is), nor does
$\tau_K$ necessarily coinside with the topology inherited by $K$ from
$\tau$. Both of these properites are readily ensured by starting with
a countable family $\K$ of compact subspaces of $X$ in some (not
necessarily $\kw$) topology $\tau'$ on $X$ and taking $\tau_K$ to be
the appropriate subspace topology.

Group $\kw$ topologies have been well studied and are useful building
blocks for various examples of sequential groups (see, for
example~\cite{Sha1} and the
references therein).

The following simple lemma demonstrates a straightforward way to build
$\kw$ group topologies (see~\cite{Shi3}, Lemma~4 for a proof of a more
general statement).

\begin{lemma}\label{kw.groups}
Let $G$ be a boolean topological group and $\K$ be a countable family
of compact subspaces of $G$ closed under finite sums. Then $\kw(\K)$
is a Hausdorff group topology on $G$. Moreover, $\kw(\K)$ is the
finest group topology on $G$ in which each $K\in\K$ remains compact.
\end{lemma}

\section{Basic definitions.}

Most known constructions of nontrivial sequential spaces use a
technique that separates the analysis of convergence from that of
the topology. Among the tools used to deal with convergence, the
standard {\em test spaces\/} (such as the sequential fan $S(\omega)$,
Arens' space $S_2$, Archangel'skii-Franklin space $S_\omega$, etc.) feature
prominently. Additional ordinal invariants, such as the
Cantor-Bendixson index, scatteredness rank
(see~\cite{Z}, \cite{Shi7}), etc.\ are often used to bound the sequential
order of `intermediate spaces' in the construction.

To outline the reasons why such methods
are of limited utility for the problem in this paper consider the
following argument.

Lemma~5 in~\cite{Shi3} states that given a $\kw$ boolean group $G$ and a
closed discrete subset $D\subseteq G$ one can find a coarser
$\kw$-topology on $G$ in which $D$ has a limit point. This suggests
the following brute force strategy for building a countably compact
sequential group that is not Fr\'echet.

Consider a subgroup $G$ of
$2^\C$ (algebraically) generated by a subspace homeomorphic to a
compact sequential space $K$ of sequential order $\geq2$ (for example, the
one point compactification of the well-known Mrowka's space
$\psi^*$). Endow $G$ with the natural $\kw$ topology determined by the
family $\K$ of all the iterated sums of $K$.

Recursively
(using an appropriate set-theoretic principle such as
$\diamondsuit$, see~\cite{Shi3} for details of similar constructions)
add new convergent sequences using the lemma from~\cite{Shi3}
mentioned above to make $G$ countably compact. The topology determined
by all the added compact subspaces together with the original family $\K$
will be sequential and countably compact providing the desired example.

It is instructive to see why the naive approach above fails. If such a
group topology $\tau$ on $G$ existed it would be easy to find a
quotient $G'$ of $G$ such that the quotient topology on $G'$ has a
countable pseudocharacter. As was shown in~\cite{Shi3} $G'$ will
then necessarily be countable (at least with the $K$ chosen
above) thus yielding a contradiction, since $G'$ must be
countably compact. In fact, it is still an open question whether it is
consistent with the axioms of ZFC that a countably compact sequential
group may contain a compact subspace of sequential order $\geq2$.

A more detailed analysis of countably compact (boolean) groups helps
to reveal the main source of difficulties with the approach above. 

Recall that a topological group $G$ is called {\em precompact\/} (or
{\em totally bounded}) if it
can be embedded as a subgroup in some compact group. The following
lemma follows from Pontryagin's duality for compact abelian groups and
the well-known characterization of precompact groups. Since this
result will never be used directly, its proof is omitted.

\begin{lemma}\label{precompact.groups}
Every countably compact group is precompact. Every precompact (thus
every countably compact) boolean group has a {\em linear} topology,
i.e.\ has a base of neighborhoods of $0$ consisting of (clopen) subgroups of
(necessarily) finite index.
\end{lemma}

The construction in~\cite{Shi3} that `forces' $D$ to aquire a
limit point adds a single convergent sequence to the original topology
of $G$ along with all of the iterated sums of such sequence.
%As
%Lemma~\ref{precompact.groups} suggests what is required instead is
%adding the whole compact (metrizable) subrgroup generated by the
%sequence. Most points of this group will lie outside of $G$.
Unless some precautions with the choice of the new
convergent sequence (as well as its limit) are taken, it may easily
destroy the 
precompactness of any topology on $G$ compatible with the
$\kw$-topology.

The construction in this paper uses two separate topologies, one to
deal with the convergence, and the other to ensure that the limit space is a topological
group: a $\kw$ topology and a coarser precompact first countable
metrizable topology (see Definition~\ref{contrip} below), respectively. In order to
keep the topologies compatible throughout the construction, instead of
using the standard test spaces to estimate the sequential order of
points, estimates of the sequential order of some homogeneous
countable subspaces are used instead (the odd subspace, see
Definitions~\ref{lase} and~\ref{U.sep}).

In the arguments below it will be convenient to consider $\kw$ groups
together with the countable family of compact subsets that determine the
topology. We therefore introduce the following shortcut.

\begin{definition}\label{contrip}
Call $(G, \K)$ a {\em $\kw$-pair (with respect to $\tau$)\/}
if $(G,\tau)$ is a boolean topological
group with the $\kw$ topology $\tau$ and $\K$ is a countable family of
compact subspaces of $G$ closed under finite sums and intersections
such that $\tau=\kw(\K)$ and $\cup\K=G$.
\end{definition}

In almost every case, the existence of the Hausdorff topology $\tau$
will be clear from the context and will not be discussed while
referring to a $\kw$-pair.

As was noted above, any inductive construction of a countably compact
group must provide a mechanism for ensuring the precompactness of the
final topology. The following definition is used throughout
and forms one of the building blocks of the
construction.

\begin{definition}\label{contrip}
Call $\gku$ a {\em convenient triple\/} if $G$ is a boolean
group, $\U$ is a countable family of subgroups closed under finite
intersections that forms an open base of neighborhoods of $0$ in some
Hausdorff precompact topology $\tau(\U)$ on $G$, and $\K$ is a
countable family of compact (in $\tau(\U)$) subgroups of $G$ closed under finite sums
and intersections such that $\cup\K=G$.
\end{definition}

Trivially, if $\gku$ is a convenient
triple then $(G, \K)$ is a $\kw$-pair (with respect to
$\kw(\K)$) and every $K\in\K$ is metrizable.

A number of arguments involve translating various subsets by compact
subspaces. The definition below lists a few ordinal invariants that
measure the effects of such shifts.

\begin{definition}\label{kdepth}
Let $(G, \K)$ be a $\kw$-pair, $A\subseteq G$.
Suppose $D\subseteq G$ is countable and there exist an
$\alpha<\omega_1$ and $K\in\K$ such that
$D\subseteq[A]_\alpha+K$. Call the smallest $\alpha$ with this
property the {\em $\K$-depth of $D$ over $A$} and write
$\hd{\K}{D}{A}=\alpha$. If no such $\alpha$ exists $\hd{\K}{D}{A}$ is
defined to be $\omega_1$.

Let $\CC=\set S^i:i\in\omega.$ be a family of subsets of $G$. Call
$m\in\omega$ the {\em $\K$-depth of $D$ over
$\CC$\/} and write $\hd{\K}{D}{\CC}=m$ if $m\in\omega$ is the smallest
such that $D\subseteq\sum_{i\leq m}\cl{\lasi}^\kwk+K$ for some
$K\in\K$. If no such $m$ exists $\hd{\K}{D}{\CC}=\omega$.

Let $0\in\cl{P}^\kwk$ for some $P\subseteq G$. Call
$m\in\omega$ the {\em asymptotic $\K$-depth of $P$ over
$\CC$\/} and write $\ha{\K}{P}{\CC}=m$ if $m\in\omega$ is the smallest
such that $0\in\cl{P\cap(\sum_{i\leq m}\cl{\lasi}^\kwk+K)}$ for some
$K\in\K$. If no such $m$ exists $\ha{\K}{P}{\CC}=\omega$.
\end{definition}

The following property is an immediate corollary of the definition above.

\begin{lemma}\label{kdshift}
Let $(G, \K)$ be a $\kw$-pair, $A\subseteq G$, and
$P,D\subseteq G$ be countable subsets. Let $P\subseteq D+K$ for some
$K\in\K$. Then $\hd{\K}{P}{A}\leq\hd{\K}{D}{A}$.
\end{lemma}

Note that adding or removing finitely many
points does not change the $\K$-depth of a set. Combined with the
lemma above it follows that $D\ain D'$ implies
$\hd{\K}{D}{A}\leq\hd{\K}{D'}{A}$.

We now define some minimality properites of sets with respect to their $\K$-depth.
\begin{definition}\label{mm.sequence}
Let $(G, \K)$ be a $\kw$-pair, $A\subseteq G$, ${\cal
A}=\set A_i:i\in\omega.$ be a family of subsets of $G$, and
$D\subseteq G$ be a countable subset. Let
$\hd{\K}{D}{A}=\delta\leq\omega_1$. Define
\begin{labeling}{\MM{A,\K}}
\item[\MM{A,\K}:]for any $\beta<\delta$, $K\in\K$ there exists a
finite $F_{K,\beta}\subseteq D$ such that
$$
(d+K)\cap[A\cup\{0\}]_\beta^\kwk=\varnothing \hbox{ if } d\in\la{D}\setminus\la{F_{K,\beta}}
$$
\end{labeling}
\begin{labeling}{\MM{A,\K}}
\item[\MM{{\cal A},\K}:]for any $n\in\omega$, $K\in\K$ there exists a
finite $F_{K,n}\subseteq D$ such that
$$
(d+K)\cap\sum_{i<n}\overline{\la{A_i}}^\kwk=\varnothing \hbox{ if }
d\in\la{D}\setminus\la{F_{K,n}}
$$
\end{labeling}
\end{definition}

By choosing $F_{K,\beta}=F_{K,n}$ above for any $K\in\K$ and any
$\beta<\omega_1$ one shows that \MM{{\cal A}, \K} implies
\MM{\la{A_n}, \K} and $\hd{\K}{D}{\la{A_n}}=\omega_1$ for every
$n\in\omega$. Also note that both properties imply that $\lad\cap K$
is finite for every $K\in\K$ so $\lad$ is closed and discrete in $\kw(\K)$.

The following lemma will be part of most `thinning out' arguments
below.

\begin{lemma}\label{free.sequence}
Let $(G, \K)$ be a $\kw$-pair, $D'\subseteq G$ be
infinite, closed and discrete in $\kw(\K)$, and let ${\cal A}=\set
A_i:i\in\omega.$ be a family of subsets of $G$.

There exists an
infinite independent $D\subseteq D'$ such that $\lad$ is closed and
discrete in $\kwk$ and $D$ satisfies \MM{A_i,\K} for every
$i\in\omega$. If $\hd{\K}{D'}{{\cal A}}=\omega$ then $D$ can be
chosen to satisfy \MM{{\cal A},\K}.
\end{lemma}
\begin{proof}
Passing to a subset if necessary, assume that $D'$ is countable.
Let $\hd{\K}{D'}{A_0}=\sigma_0\leq\omega_1$. Suppose
$0<\sigma_0\leq\omega_1$ and pick $\sigma_n^0\to\sigma_0$ if
$\sigma_0<\omega_1$ is a limit ordinal and $\sigma_n^0=\sigma_0-1$
otherwise. If $\sigma_0=\omega_1$ let $\sigma_n^0=\omega_1$.

Select points $d_n\in D'$ so that
$$
d_n\not\in\sum_{i<n}K_i+\langle\set
d_i:i<n.\rangle+[A_0\cup\{0\}]_{\sigma_n^0}^\kwk
$$
If the recursion terminates at some $n\in\omega$ then
$D'\subseteq[A_0\cup\{0\}]_{\sigma_n^0}+K$ for some $K\in\K$. If
$\sigma_n^0<\omega_1$ then $\hd{\K}{D'}{A_0}\leq\sigma_n^0<\sigma_0$
contradicting the choice of $\sigma_0$. If $\sigma_n^0=\omega_1$ then
$\sigma_0=\omega_1$. Since $D'$ is countable there is a
$\sigma'<\omega_1$ such that $D'\subseteq[A_0]_{\sigma'}^\kwk+K$
contradicting $\hd{\K}{D'}{A_0}=\omega_1$. 

Let $D_0=\set d_n:n\in\omega.$ and $D\ain D_0$. Let
$K=K_{i'}\in\K$, $\beta<\sigma_0$, 
$\sigma_m^0\geq\beta$. By picking $i\geq i'$ large
enough if necessary we may assume that $\la{D\setminus D_0}+K\subseteq
\sum_{i'<i}K_{i'}$. Let $n>\max\{i,m\}$. Put $F_K=\set d_j:j<n.$ and let $d\in\lad\setminus\la{F_K}$. Then
$d=d_{n'}+d'+a$ where $n'\geq n$, $a\in \la{D\setminus D_0}$, and
$d'\in\la{\set d_j:j<n'.}$. Now
$d_{n'}\not\in \sum_{i'<i}K_{i'}+d'+[A_0\cup\{0\}]_\beta^\kwk$ so $d\not\in
K+[A_0\cup\{0\}]_\beta^\kwk$. Thus $\lad\cap K$ is finite for
every $K\in\K$ so $\lad$ is closed in $\kwk$,
$\hd{\K}{D}{A_0}=\sigma_0$, and \MM{A_0, \K} holds for every infinite
$D\ain D_0$. The argument for $\sigma_0=0$ is similar,
replacing $[A_0\cup\{0\}]_{\sigma_n^0}^\kwk$ with $\{0\}$. If
$\hd{\K}{D}{{\cal A}}=\omega$ replace
$[A_0\cup\{0\}]_{\sigma_n^0}^\kwk$ with $\sum_{i\leq
n}\cl{\la{A_i}}^\kwk$. 

Repeatedly using the construction above construct
$D_0\supseteq D_1\supseteq\cdots\supseteq D_n\supseteq\cdots$ such
that $\la{D_i}$ is closed and discrete in $\kwk$ for every
$i\in\omega$ and \MM{A_i, \K} holds for every infinite $D\ain D_i$.
Let $D\ain D_i$ for every $i\in\omega$. If $D_0$
satisfies \MM{{\cal A},\K} put $D=D_0$.
\end{proof}

The remark after Definition~\ref{mm.sequence} now gives the following
corollary (after selecting a trivial $A=\{0\}$).

\begin{corollary}\label{d.subgroup}
Let $(G,\K)$ be a $\kw$-pair, $D\subseteq G$ be a countable closed and
discrete subspace of $G$. Then there exists an infinite $D'\subseteq
D$ such that $\la{D'}$ is closed and discrete in $G$.
\end{corollary}

\begin{lemma}\label{akdepth}
Let $(G, \K, \U)$ be a convenient triple and let ${\cal A}=\set
A_i:i\in\omega.$ be a family of subsets of $G$. Let $P\subseteq G$ be
a countable subset such that $0\in\cl{P}^{\tau(\U)}$ and $\hd{\K}{P\cap U}{{\cal
A}}=\omega$ for every $U\in\U$ (in particular, if $\ha{\K}{P}{{\cal
A}}=\omega$). 

Then there exists an $S\subseteq P$ such that $S\to0$ in $\tau(\U)$
and $S$ satisfies \MM{{\cal A}, \K}.
\end{lemma}
\begin{proof}
Let $\K=\set K_n:n\in\omega.$ and $\U=\set U_n:n\in\omega.$. Select
points $p_n\in P\cap\cap_{i\leq n}U_i$ so that
$$
p_n\not\in\sum_{i\leq n}K_i+\la{\set p_i:i<n.}+\sum_{i\leq
n}\cl{\la{A_i}}^\kwk
$$
Since $\hd{\K}{P\cap U}{{\cal A}}=\omega$ for every $U\in\U$ the
choice of $p_n$ is always possible. Verifying that $S=\set
p_n:n\in\omega.\to0$ in $\tau(\U)$ is routine.
The proof that $S$ satisfies \MM{{\cal A}, \K} is is the same as
in Lemma~\ref{free.sequence}.

To see that the condition of the lemma follows from $\ha{\K}{P}{{\cal
A}}=\omega$, let $m\in\omega$, $K\in\K$, and $U\in\U$. Then $\ha{\K}{P}{{\cal
A}}=\omega$ implies
$$0\not\in\cl{P\cap(\sum_{i\leq m}\cl{\la{A_i}}^\kwk+K)}^\kwk$$
therefore $P\cap
U\not\subseteq\sum_{i\leq m}\cl{\la{A_i}}^\kwk+K$.
Thus $\hd{\K}{P\cap U}{{\cal A}}=\omega$.
\end{proof}

\section{`Convexity' and extensions in boolean groups}
Given a convenient triple $\gku$, a common operation is to extend $G$ and $\K$
by adding points from the compact completion of $(G,\tau(\U'))$ where
$\U'$ extends $\U$. The
convergence properties of the new convenient triple $\gkup$ will
depend on the precompact topology whose extension produces $G'$. This
section lists a few results to help control these properties.

The following lemma may be viewed as a boolean
$\kw$ version of the classical Hahn-Banach theorem.
\begin{lemma}\label{bhb}
Let $(G, \K)$ be a $\kw$-pair such that every $K\in\K$ is a (compact)
subgroup of $G$, let $H$ be a subgroup of $G$ closed in
$\kwk$, and $g\in G\setminus H$. Then there exists a subgroup
$U\subseteq G$ of finite index open in $\kwk$ such that $H\subseteq U$
and $g\not\in U$.
\end{lemma}
\begin{proof}
Let $\K=\set K_i:i\in\omega.$. Build by induction a sequence of closed
subgroups $H\subseteq H_0\subseteq\cdots\subseteq H_n\subseteq\cdots$
such that $H_{n+1}=H_n+K_{n+1}'$ for some subgroup
$K_{n+1}'\subseteq K_{n+1}$ clopen in $K_{n+1}$, and $g\not\in H_n$. If $H_n$ has been built,
$g\not\in H_n$ so there exists a clopen subgroup $K_{n+1}'\subseteq
K_{n+1}$ such that $g\not\in H_{n+1}=H_n+K_{n+1}'$. Then $H_{n+1}$ is
a closed subgroup of $G$ and $H_{n+1}\cap K_i$ is relatively open in
$K_i$ for every $i\leq n+1$.

Let $H'=\cup_{n\in\omega}H_n$. Now the intersection $H'\cap K$ is relatively open in $K$ for every
$K\in\K$ so $H'$ is an open subgroup of $G$ in
$\kwk$. Let $G'\subseteq G$ be a subgroup such that
$G'\cap\la{H'\cup\{g\}}=\{0\}$ and $G'+\la{H'\cup\{g\}}=G$. Then
$g\not\in U=H'+G'$ is open in $\kwk$ and of finite index.
\end{proof}

In the statement of the next lemma we abuse the notation to use
$\tau(\U_0)$ for both the topology of the compact completion of $G$,
as well as the topology on $G$ generated by $\U_0$.
\begin{lemma}\label{ct.resolve}
Let $\gku$ be a convenient triple, let $H$ be a subgroup of $G$ closed in
$\kwk$. Then there exists a countable family of open (in $\kw(\K)$)
subgroups of finite index $\U_0\supseteq\U$ such that
$\cl{H}^{\tau(\U_0)}\cap G=\cap\set U\in\U_0:H\subseteq U.=H$.
\end{lemma}
\begin{proof}
Let $\K=\set K_n:n\in\omega.$. For each $K\in\K$ let $\set
U_n^K\subseteq K:n\in\omega.$ be a countable base of neighborhoods of
$0$ in $K$ consisting of clopen subgroups of finite index. The family
$\K_n^K=\set U_n^K+a:a\in K.$ is finite so
$\K^+=\cup_{m,n\in\omega}\K_n^{K_m}$ is countable. For each
$K\in\K^+$, let $U(K)\subseteq G$ be an open in $\kwk$ subgroup of finite index
such that $H\subseteq U(K)$ and $H\cap K=\varnothing$ if
such $U(K)$ exists. Let $U(K)=G$ otherwise. Let $\U_0$ be the closure
of $\U\cup\set U(K):K\in\K^+.$ under finite intersections.

Let $g\in G\setminus H$. By Lemma~\ref{bhb} there exists an open in
$\kwk$ subgroup $U$ of finite index such that $H\subseteq U$ and
$g\not\in U$. Let $K'\in\K$ be such that $g\in K'$ and let $U_n^{K'}$
be such that $U_n^{K'}\subseteq K'\cap U$. Then $K=g+U_n^{K'}\in\K^+$ and
$U\cap K=\varnothing$ so
$H\subseteq U(K)$, $U(K)\cap K=\varnothing$. Since $U(K)\in\U_0$ is clopen,
$g\not\in\cl{H}^{\tau(\U_0)}$.
\end{proof}

\begin{lemma}\label{ct.split}
Let $\gku$ be a convenient triple, let $S$ be a countable
independent subset of $G$, closed and
discrete in $\kwk$. Then there exist a countable family of open
subgroups of finite index $\U_0\supseteq\U$ such that
$\cl{\la{S'}}^{\tau(\U_0)}\cap G=\la{S'}$ for any $S'\subseteq S$ and
$\cl{\lase}^{\tau(\U_0)}\cap\cl{\laso}^{\tau(\U_0)}=\varnothing$.
\end{lemma}
\begin{proof}
Note that $\la{S'}=\cap\set \la{S\setminus\{s\}}:s\in S\setminus S'.$
and $s\not\in\lase$ for any $s\in S$. Now apply Lemma~\ref{ct.resolve}
repeatedly to find $\U_0\supseteq\U$ such that
$\cl{\la{S\setminus\{s\}}}^{\tau(\U_0)}\cap G=\la{S\setminus\{s\}}$ for
every $s\in S$ and $\cl{\lase}^{\tau(\U_0)}\cap G=\lase$.
\end{proof}

We now define a basic extension operation used in the construction.

\begin{definition}\label{pse}
Let $(G, \K, \U)$ and $(G', \K', \U')$ be convenient triples and an
independent $D\subseteq G$ be such that $\lad$ is closed and discrete
in $\kwk$. Call $(G', \K', \U')$ a {\em primitive sequential
extension (pse for short) of $(G, \K, \U)$ over $D$\/} if the
following conditions hold:
\begin{countup}[series=bcp]
\item\label{pse.order}
$G\subseteq G'$, $\cl{U}^{\tau(\U')}\in\U'$ for every $U\in\U$, and
$\K'$ is the closure of $\K\cup\{L\}$ under finite sums where
$L=\cl{\lad}^{\tau(\U')}$ (thus $L$ is compact in $\kw(\K')$);

\item\label{pse.resolve}
$\cl{\la{D\setminus F}}^{\tau(\U')}\cap G=\la{D\setminus
F}$ for any $F$;

\end{countup}
\end{definition}
Abusing the notation we will also use $\tau(\U')$ to refer to the
topology on $G$ induced by $\tau(\U')$.

If the new points added to $G$ are `sufficiently far' from some set
$A\subseteq G$, the convergence properties at the points `near $A$'
may not be affected.
\begin{lemma}\label{D.flat}
Let $(G', \K', \U')$ be a pse of $(G, \K, \U)$ over $D\subseteq
G$. Let $A\subseteq G$, $\hd{\K}{D}{A}=\delta\leq\omega_1$, 
and let $D$ satisfy~\MM{A,\K}. If $\so(a,
A, \kw(\K'))=\sigma\leql\delta$ then $a\in G$ and $\so(a,
A, \kwk)=\sigma$.
\end{lemma}
\begin{proof}
If $a\in[A]_0$ then $a\in A\subseteq G$. Suppose the Lemma has been proved
for all $\sigma'<\sigma$. Pick $a_n\to a$ (in $\kw(\K')$) such that
$\so(a_n, A, \kw(\K'))=\sigma_n$ for some increasing
$\sigma_n\to\sigma-1$. Then by the inductive hypothesis $a_n\in G$ and
$\so(a_n, A, \kwk)=\sigma_n$.

Since $a_n\to a$ in $\kw(\K')$ there is a $K'\in\K$ such that
$\set a_n:n\in\omega.\subseteq K'+L$ where
$L=\overline{\lad}^{\tau(\U')}$. By thinning out and reindexing
we may assume that $a_n=a^n+d^n$ where $a^n\in K'$, $a^n\to a'\in K'$
and $d^n\in L$, $d^n\to d\in L$. Since $a^n,a_n\in G$,
by~\ref{pse.resolve} $d^n\in\lad$.

Let $\beta=\sigma$ if $\delta$ is limit or $\beta=\delta-1$
otherwise. Then $\beta<\delta$ and $\sigma_n\leq\beta$ for every
$n\in\omega$. Using~\MM{A, \K} pick a finite $F\subseteq D$ such
that $(d+K')\cap[A]_\beta=\varnothing$ for every
$d\not\in\laf$. Then $(d+K')\cap[A]_{\sigma_n}=\varnothing$ for
every $n\in\omega$ and $d\not\in\laf$.

Now $d^n\in\laf$ for every $n\in\omega$. We may assume that $d^n=d$
for every $n\in\omega$ so $a_n=a^n+d\in K''$ for some $K''\in\K$ and
$a_n\to a$ in $\kwk$. It follows that $a\in G$ and $\so(a,
A, \kwk)\leq\sigma$. Since $\kw(\K)\subseteq\kw(\K')$, $\so(a,
A, \kwk)=\sigma$.
\end{proof}

\section{Parity and separation}
While Lemma~\ref{D.flat} provides one way for preserving the
sequential order at some points, it is not always possible to expect a
given set to be far from a fixed witness to the sequential order. A
different mechanism is needed, introduced in this section.

The next definition is a convenient way to set a lower bound on the
sequential order in a $\kw$ group. Note that it is not required that
$0\in[S]_\sigma^\kwk$ (or $0\in\cl{S}^\kwk$).
\begin{definition}\label{U.sep}
Let $(G, \K, \U)$ be a convenient triple and $S\subseteq G$. Let
$\sigma$ be a successor. Say that $\S(\K, S, \sigma)$ holds if for every
$K\in\K$ and every $\sigma'<\sigma-1$ there exists a clopen
subgroup $\UU_K(\K, S, \sigma')\subseteq K$ such that $\UU_K(\K,
S, \sigma')\cap[\laso]_{\sigma'}=\varnothing$.
\end{definition}

\begin{lemma}\label{U.sep.so}
Let $(G, \K, \U)$ be a convenient triple and $S\subseteq G$ be a
countable independent subset. Let $\sigma<\omega_1$ be a successor
ordinal. Then $\S(\K, S, \sigma)$ holds if and only if
$\so(0, \laso, \kwk)\geq\sigma$.
\end{lemma}
\begin{proof}
Suppose $\so(0, \laso, \kwk)=\sigma'<\sigma$. Then there exist
$s_n\to0$ in $\kw(\K)$ such that
$\so(s_n, \laso, \kwk)=\sigma_n\leq\sigma'-1$. By taking a subsequence
if necessary we may assume that $s_n\in K$ for some $K\in\K$. Since
$\sigma_n\leq\sigma'-1<\sigma'\leq\sigma-1$, by $\S(\K, S, \sigma)$ there exists a
clopen subgroup $U=\UU_K(\K, S, \sigma'-1)\subseteq K$ such that
$U\cap[\laso]_{\sigma'-1}=\varnothing$. Since $0\in U$ and
$s_n\in[\laso]_{\sigma'-1}$, this contradicts $s_n\to0$. Hence $\S(\K,
S, \sigma)$ implies $\so(0, \laso, \kwk)\geq\sigma$.

The converse will not be used so its proof is
omitted. Its proof uses the property that each $K\in\K$ is first
countable.
\end{proof}

The algebraic tool used to control the convergence properties in the
extension is given by the following {\em parity\/} homomorphism.
\begin{definition}\label{pre.parity}
Let $(G', \K', \U')$ be a pse of $(G, \K, \U)$ over $D\subseteq G$, and let
$L=\overline{\lad}^{\tau(\U')}$. Define
$\Pty^S:[\las]_\sigma^\kwk\to2$ and $\Pty^L:L\to2$ by letting
$\Pty^S(a)=0$ if $a\in[\lase]_\sigma^\kwk$ and $\Pty^S(a)=1$ if
$a\in[\laso]_\sigma^\kwk$. If $d\in L$ put $\Pty^L(d)=0$ if
$d\in\overline{\lade}^{\tau(\U')}$ and $\Pty^L(d)=1$ otherwise. If $b=a+d$ for
some $a\in[\las]_\sigma^\kwk$ and $d\in L$ put
$\Pty(b)=\Pty^S(a)+\Pty^L(d)$.
\end{definition}
The next lemma shows that the parity homomorphism is well defined
under some conditions.
\begin{lemma}\label{parity}
Let $(G', \K', \U')$ be a pse of $(G, \K, \U)$ over
$D\subseteq[\laso]_\sigma^\kwk$ where $S\subseteq G$ is an
independent set, and let
$L=\overline{\lad}^{\tau(\U')}$. If
$\overline{\lade}^{\tau(\U')}\cap\overline{\lado}^{\tau(\U')}=\varnothing$ and
$\S(\K,S,\sigma+1)$ holds then $\Pty:[\las]_\sigma^\kwk+L\to2$ is
a well defined homomorphism, continuous on $L$.
\end{lemma}
\begin{proof}
Since $\las=\lase\cup(\laso)$ if $a\in[\las]_\sigma^\kwk$ for
some $\sigma<\omega_1$ then either $a\in[\laso]_\sigma^\kwk$ or
$a\in[\lase]_\sigma^\kwk$. If
$a\in[\laso]_\sigma^\kwk\cap[\lase]_\sigma^\kwk$ then
$a\in[\laso]_{\sigma'}^\kwk\cap[\lase]_{\sigma'}^\kwk$ for some
successor ordinal $\sigma'\leq\sigma$. Thus there are
$a_n^0\to a$ and $a_n^1\to a$ such that $a_n^0\in[\lase]_{\sigma_n^0}$
and $a_n^1\in[\laso]_{\sigma_n^1}$ for some
$\sigma_n^0, \sigma_n^1\leq\sigma'-1<\sigma$. Now $a_n^0+a_n^1\to0$ and
$a_n^0+a_n^1\in[(\laso)+\lase]_{\sigma'-1}^\kwk=[\laso]_{\sigma'-1}^\kwk$
by Lemma~\ref{set.additivity}. Pick a $K\in\K$ such that $\set
a_n^0+a_n^1:n\in\omega.\subseteq K$. Then $\UU_K(\K,
S, \sigma'-1)\cap\set a_n^0+a_n^1:n\in\omega.=\varnothing$
contradicting $a_n^0+a_n^1\to0$.

Thus $\Pty^S$ is well defined
on $[\las]_\sigma^\kwk$. A similar argument involving
Lemma~\ref{set.additivity} and $(\laso)+(\laso)=\lase$ shows that
$\Pty^S$ is a homomorphism.

It follows from the definition of $\Pty^L$ and the choice of $D$ that
$\Pty^L$ is a continuous homomorphism on $L$.

If $d\in L\cap[\las]_\sigma^\kwk$ then $d\in\lad$ by~\ref{pse.resolve} so
$\Pty^S(d)=\Pty^L(d)$ by Lemma~\ref{set.additivity} and
$D\subseteq[\laso]_\sigma^\kwk$. Let
$a+d=a'+d'\in[\las]_\sigma^\kwk+L$ where $a,a'\in[\las]_\sigma^\kwk$ and
$d,d'\in L$. Then $a+a'=d+d'$ so $d+d'\in G$ and thus
$d+d'\in\lad\subseteq[\laso]_\sigma^\kwk$ by~\ref{pse.resolve}
and Lemma~\ref{set.additivity}. Now
$\Pty^S(a)+\Pty^S(a')=\Pty^S(a+a')=\Pty^S(d+d')=\Pty^L(d+d')=\Pty^L(d)+\Pty^L(d')$
so $\Pty^S(a)+\Pty^L(d)=\Pty^S(a')+\Pty^L(d')$ and $\Pty$ is well
defined.
\end{proof}
While the parity homomorphism is unlikely to be continuous on its
domain (even if an appropriate topology is agreed upon) it satisfies
the following weak continuity property.

\begin{lemma}\label{pparity}
Let $(G', \K', \U')$ be a pse of $(G, \K, \U)$ over
$D\subseteq G$. Let $D$ satisfy~\MM{\laso,\K} and
$\overline{\lade}^{\tau(\U')}\cap\overline{\lado}^{\tau(\U')}=\varnothing$,
and let $L=\overline{\lad}^{\tau(\U')}$. Let
$\hd{\K}{D}{\laso}=\delta\leq\omega_1$ and
$D\subseteq[\laso]_{\delta}^\kwk$ if $\delta<\omega_1$. Let $\alpha$
be a successor and $\S(\K,S,\alpha)$ hold. If $\delta<\alpha$, and
$\so(g, \laso, \kw(\K'))=\sigma<\alpha$ then
$g\in[\las]_\sigma^\kwk+L$, $\Pty(g)=1$ where
$\Pty:[\las]_{\alpha-1}^\kwk+L\to2$ is the homomorphism in
Definition~\ref{pre.parity}.
\end{lemma}
\begin{proof}
Since $\delta<\alpha$ the homomorphism
$\Pty:[\las]_{\alpha-1}^\kwk+L\to2$ is well defined by
Lemma~\ref{parity}.

Suppose the Lemma has been proved for all $g'\in G$, $\sigma'<\sigma$
such that $\so(g', \laso, \kw(\K'))=\sigma'$ and let
$\so(g, \laso, \kw(\K'))=\sigma$.

If $\sigma\leql\delta$ then $\so(g, \laso, \kwk)=\sigma$ by
Lemma~\ref{D.flat}. Since $\sigma$ is a successor, assume below that
$\sigma>\delta$ and let $g_n\to g$ in $\kw(\K')$ be such that
$\so(g_n, \laso, \K')=\sigma_n$ where $\sigma_n\to\sigma-1$ and
$\sigma_n<\sigma$.  

Applying the inductive hypothesis each $g_n\in[\las]_{\sigma_n}^\kwk+L$, and
$\Pty(g_n)=1$ so there exist $a_n\in[\las]_{\sigma_n}^\kwk$ and $d_n\in L$
such that $g_n=a_n+d_n$. Since $g_n\to g$ in $\kw(\K')$ there
are (possibly after thinning out and reindexing) a $K'\in\K$, $a^n\in
K'$, and $d^n\in L$ such that $a^n\to a$, $d^n\to d$ and
$g_n=a^n+d^n$. Now $a^n-a_n=d_n-d^n=d_n'$ so $d_n'\in\lad$
by~\ref{pse.resolve}. Thus $a^n=a_n+d_n'$ where
$d_n'\in[\las]_{\delta}^\kwk$. Therefore
$a^n\in[\las]_{\sigma_n'}^\kwk$ for some $\sigma_n'<\sigma$ by
$\delta<\sigma$ and Lemma~\ref{set.additivity}.

After picking a subsequence and reindexing, assume that
$\Pty(a^n)$ and $\Pty(d^n)$ are constant. Then $\Pty(a)=\Pty(a^n)$,
$\Pty(d)=\Pty(d^n)$ by the definition of $\Pty$ and Lemma~\ref{parity}
so $\Pty(g)=\Pty(a+d)=\Pty(a_n+d_n)=\Pty(g_n)=1$ by Lemma~\ref{parity}.
\end{proof}

The main reason the parity homomorphism was defined is the proof of
the following lemma that states the conditions under which the
sequential order of some points is preserved across primitive
sequential extensions.

\begin{lemma}\label{D.bind}
Let $\gku$ be a convenient triple, $S\subseteq G$ be such that $\S(\K,
S, \alpha)$ holds for some successor ordinal $\alpha<\omega_1$. Let
$\gkup$ be a pse of $\gku$ over $D\subseteq G$ that has the following
properties.
\begin{countup}[bcp]
\item\label{bind.proximity}
$D$ satisfies \MM{\laso, \K}, $\hd{\K}{D}{\laso}=\delta$ and either
$\delta\geq\alpha-1$ or $D\subseteq[\laso]_\delta^\kwk$;

\item\label{bind.split}
if $\delta<\alpha-1$ then
$\cl{\lado}^{\tau(\U')}\cap\cl{\lade}^{\tau(\U')}=\varnothing$;

\end{countup}
Then $\S(\K', S, \alpha)$ holds so that $\UU_K(\K',
S, \sigma)=\UU_K(\K, S, \sigma)$ for every $K\in\K$ and
$\sigma<\alpha-1$.
\end{lemma}
\begin{proof}
Let $L=\cl{\lad}^{\tau(\U')}$, $K\in\K$, and $\sigma<\alpha-1$.

Suppose $\sigma<\delta$. Pick a finite $F_K\subseteq D$ using~\MM{\laso, \K} such
that $(K+d)\cap[\laso]_\sigma^\kwk=\varnothing$ for every
$d\in\lad\setminus\la{F_K}$. Put $U_{K+L}(\sigma)=K+\cl{\la{D\setminus
F_K}}^{\tau(\U')}$.

Suppose
$U_{K+L}(\sigma)\cap[\laso]_\sigma^{\kw(\K')}\not=\varnothing$. Let
$u\in K$, $d\in\cl{\la{D\setminus F_K}}^{\tau(\U')}$, and $a\in G'$ be such that
$u+d=a$ where $a\in[\laso]_\sigma^{\kw(\K')}$. By
Lemma~\ref{D.flat} $a\in[\laso]_\sigma^\kwk$ so $d\in G$. Then
$d\in\la{D\setminus F_K}$ by~\ref{pse.resolve} and
$(K+d)\cap[\laso]_\sigma^\kwk=\varnothing$ by the choice of $F_K$
contradicting the choice of $a$.

Suppose $\sigma\geq\delta$. Then $\delta<\alpha-1$. Put
$U_{K+L}(\sigma)=\UU_K(\K,
S, \sigma)+\overline{\lade}^{\tau(\U')}$. Then $U_{K+L}(\sigma)$ is a
compact subgroup of finite index in $K+L$ and thus clopen in
$K+L$. Suppose
$U_{K+L}(\sigma)\cap[\laso]_\sigma^{\kw(\K')}\not=\varnothing$. Then
there exist $u\in\UU_K(\K, S, \sigma)$ and
$d\in\overline{\lade}^{\tau(\U')}$ such that $u+d=a$ for some
$a\in[\laso]_\sigma^{\kw(\K')}$. By Lemma~\ref{parity}
and~\ref{bind.split} $\Pty:[\las]_\sigma^\kwk+L\to2$ is well defined
and $a=a'+d'$ where $a'\in[\las]_\sigma^\kwk$, $d'\in L$ and
$\Pty(a'+d')=1$ by Lemma~\ref{pparity}.

Now $u=a+d$ thus $a+d\in G$. Since $\Pty(a'+d')=1$ by
Lemma~\ref{pparity} and $\Pty(d)=0$ by Definition~\ref{pre.parity},
$\Pty(a'+d'+d)=1$. Since $a+d\in G$ and $u=a+d=a'+d'+d$, $d'+d\in
L\cap G=\lad\subseteq[\las]_\sigma^\kwk$ by~\ref{pse.resolve} and
$\delta\leq\sigma$.

Since $a'+d'+d\in[\las]_\sigma^\kwk$ by Lemma~\ref{set.additivity} and
$\Pty^S(a'+d'+d)=\Pty(a'+d'+d)=1$, $a'+d'+d\in[\laso]_\sigma^\kwk$ by the
definition of $\Pty^S$ contradicting the choice of $\UU_K(\K,
S, \sigma)$.

Now for every $K\in\K'$ pick a $p(K)\in\K$ such that $K\subseteq
p(K)+L$ and $p(K)=K$ if $K\in\K$. Note that $\UU_{p(K)}(\K,
S, \sigma)\subseteq U_{p(K)+L}(\sigma)$ for $K\in\K$. Put $\UU_K(\K',
S, \sigma)=U_{p(K)+L}(\sigma)\cap K$ if $K\in\K'\setminus\K$ and $\UU_K(\K',
S, \sigma)=\UU_K(\K, S, \sigma)$ otherwise.
\end{proof}

Let $\gamma$ be an ordinal. Suppose for every $\sigma<\gamma$ a
convenient triple $\hku{^\sigma}$ is defined so that the following
conditions hold:

\begin{countup}[bcp]
\item\label{stack.order}
$H^{\sigma'}\subseteq H^\sigma$, $\K^{\sigma'}\subseteq \K^\sigma$,
and $\U^{\sigma'}\subseteq \set U\cap H^{\sigma'}:U\in\U^\sigma.$
if $\sigma'\leq\sigma<\gamma$;

\item\label{stack.dense}
$H^{\sigma'}$ is dense in $H^\sigma$ in $\kw(\K^\sigma)$ for every
$\sigma'\leq\sigma$;

\end{countup}
Define $\hku{^{<\gamma}}$ by taking
$H^{<\gamma}=\cup_{\sigma<\gamma}H^\sigma$,
$\K^{<\gamma}=\cup_{\sigma<\gamma}\K^\sigma$,
$\U^{<\gamma}=\set \cl{U}^{\kw(\K^{<\gamma})}:U\in\U^\sigma, \sigma<\gamma.$.

Note that in the case of a successor $\gamma$, $\hku{^{<\gamma}}=\hku{^{\gamma-1}}$.
\begin{lemma}\label{stack.limit}
The family $\U^{<\gamma}$ forms a base of clopen subgroups of finite
index for a precompact group
topology $\tau(\U^{<\gamma})$ on $H^{<\gamma}$ and each $H^\sigma$, $\sigma<\gamma$ is dense
in $H^{<\gamma}$ in $\kw(\K^{<\gamma})$. If $\gamma<\omega_1$ then
$\hku{^{<\gamma}}$ is a convenient triple.
\end{lemma}
\begin{proof}
Let $U\in\U^{\sigma'}$ for some $\sigma'<\gamma$. Then $U+F=H^{\sigma'}$ for
some finite $F\subseteq H^{\sigma'}$. Thus
$\cl{U}^{\kw(\K^\sigma)}+F=H^\sigma$ for any
$\sigma'\leq\sigma<\gamma$. Therefore,
$\cl{U}^{\kw(\K^{<\gamma})}+F=H^{<\gamma}$ so $\cl{U}^{\kw(\K^{<\gamma})}$ is a
clopen subgroup of finite index in $H^{<\gamma}$.

Furthermore, if $g\in H^{<\gamma}\setminus\{0\}$ then $g\in
H^{\sigma'}$ for some $\sigma'<\gamma$ and there exists a
$U'\in\U^{\sigma'}$ such that $g\not\in U'$. If $\sigma\geq\sigma'$ then
by~\ref{stack.order} $U'=U\cap H^{\sigma'}$ for some
$U\in\U^\sigma$. Thus $g\not\in\cl{U'}^{\kw(\K^\sigma)}$ for any
$\sigma<\gamma$. Now
$g\not\in\cl{U'}^{\kw(\K^{<\gamma})}=\cup_{\sigma<\gamma}\cl{U'}^{\kw(\K^\sigma)}$. Thus
$\tau(\U^{<\gamma})$ forms a base of a precompact $T_1$ group topology
on $H^{<\gamma}$ that consists of clopen subgroups. The rest of the
properties are routine.
\end{proof}

The lemma below shows that iterated extensions preserve the sequential
order some points have when they are added to the group.
{%
\def\alphaDm{\alpha_{\gamma'}^m}%
\def\lad{\langle \Dg\rangle}%
\def\Dg{D_{\gamma'}}%
\begin{lemma}\label{psestack}
Let $\gku$ be a convenient triple, $S\subseteq G$ be a countable
independent subset, and
$\delta<\omega_1$ be a successor ordinal. Suppose the family
$\set\hku{^\gamma}:\gamma<\alpha.$ where $\alpha<\omega_1$ has the
following properties.
\begin{countup}[bcp]
\item\label{ipse.order}
$\hku{^\gamma}$ is a primitive sequential extension
of $\hku{^{<\gamma}}$ over some $D_\gamma\subseteq H^{<\gamma}$ for
every $0\leq\gamma<\alpha$ and $\hku{^{<0}}=\gku$;

\item\label{ipse.sep}
if $\hd{\K^{<\gamma}}{D_\gamma}{\laso}=\delta_\gamma$ then either
$\delta_\gamma\geq\delta-1$ or $D_\gamma\subseteq[\laso]_{\delta_\gamma}^{\kw(\K^{<\gamma})}$,
$D_\gamma$ satisfies~\MM{\laso, \K^{<\gamma}}, and
$\cl{\lo{D_\gamma}}^{\tau(\U^\gamma)}\cap\cl{\le{D_\gamma}}^{\tau(\U^\gamma)}=\varnothing$;

\item\label{ipse.par}
$\S(\K^\gamma, S, \delta)$ holds for every $\gamma<\alpha$;

\end{countup}
Suppose $g\in H^\gamma\setminus H^{<\gamma}$ for some $\gamma<\alpha$
and $\so(g, \laso, \kw(\K^{<\alpha}))=\sigma<\delta$. Then
$\so(g, \laso, \kw(\K^\gamma))=\sigma$.
\end{lemma}
\begin{proof}
Suppose the statement holds for all $g'\in H^{<\alpha}$ such that
$\so(g', \laso, \kw(\K^{<\alpha}))=\sigma'<\sigma$ for some
$\sigma<\delta$, and let $g\in H^\gamma\setminus H^{<\gamma}$ be such that
$\so(g, \laso, \kw(\K^{<\alpha}))=\sigma$. Let $\gamma'$ be the
smallest ordinal with the following property. There exist $g_n\to g$
in $\kw(\K^{\gamma'})$ such that
$\so(g_n, \laso, \kw(\K^{<\alpha}))=\sigma_n<\sigma$ such that
$\sigma_n\to\sigma-1$, $g_n\in H^{\gamma_n}\setminus H^{<\gamma_n}$,
and $\gamma_n\leq\gamma'$ is increasing.

Note that $g_n$ as above exist by the definition of
$\so(g, \laso, \kw(\K^{<\alpha}))$. By $g_n\to g$ there exists a
$K\in\K^\beta$ for some $\beta<\alpha$ such that $g_n\in K$ for
every $n\in\omega$. Thus $\gamma_n\leq\beta$ and
$\gamma'\leq\beta<\alpha$ is well defined.

Suppose $\gamma'>\gamma$. Then $\hku{^{\gamma'}}$ is a primitive
sequential extension of $\hku{^{<\gamma'}}$ over some $\Dg\subseteq
H^{<\gamma'}$ such that $\hd{\K^{<\gamma'}}{\Dg}{\laso}=\delta_{\gamma'}$.

If $\sigma\leql\delta_{\gamma'}$ then by Lemma~\ref{D.flat}
$\so(g, \laso, \kw(\K^{<\gamma'}))=\sigma$. Otherwise, since
$\S(\K^{\gamma'}, S, \delta)$ holds and $\delta_{\gamma'}<\sigma<\delta$,
$\Pty:[\las]_{\sigma}^{\kw(\K^{<\gamma'})}+\overline{\lad}\to2$ is
defined by Lemma~\ref{parity}. By Lemma~\ref{pparity} $g=a+s$ where
$\so(a, \las, \kw(\K^{<\gamma'}))\leq\sigma$,
$s\in\overline{\lad}^{\tau(\U^{<\gamma'})}$, and $\Pty(a+s)=1$. Since
$g\in H^\gamma\subseteq
H^{<\gamma'}$, $s\in H^{<\gamma'}$ so by~\ref{pse.resolve}
$s\in\lad$. By Lemma~\ref{set.additivity}
$\so(a+s, \las, \kw(\K^{<\gamma'}))=\max\{\sigma, \delta_{\gamma'}\}=\sigma$. Since
$\Pty(a+s)=1$, $\so(g, \laso, \kw(\K^{<\gamma'}))=\sigma$.

Thus we may assume that $g_n$ are chosen so that $g_n\to g$ in
$\kw(\K^{<\gamma'})$. Let $K\in\K^\beta$ for some $\beta<\gamma'$ be
such that $g_n\in K$. Then $g_n\to g$ in $\kw(\K^\beta)$ and
$\gamma_n\leq\beta$ contradicting the choice of $\gamma'$.

Thus $\gamma'\leq\gamma$ and $g_n\in H^\gamma$ for every $n\in\omega$
so by the hypothesis $\so(g_n, \laso, \kw(\K^\gamma))\leq\sigma_n$ and
$\so(g, \lasom, \kw(\K^\gamma))=\sigma$.
\end{proof}
}%

The next lemma shows that not only is the sequential order preserved,
it is preserved in a `uniform' way if certain conditions are met.

\begin{lemma}\label{wparity}
Let $(G, \K, \U)$ be a convenient triple and $S\subseteq G$ be a
countable independent subset. Suppose
the family $\set\hku{^\gamma}:\gamma<\alpha.$, $\alpha\leq\omega_1$
has the following properties:
\begin{countup}[bcp]
\item\label{wparity.order}
conditions~\ref{ipse.order} and~\ref{ipse.sep} of Lemma~\ref{psestack}
hold;

\item\label{wparity.par}
$\S(\K^\gamma, S, \delta)$ holds so that $\UU_K(\K^\gamma,
S, \sigma)=\UU_K(\K^\beta, S, \sigma)$ for any $K\in\K^\beta$,
$\sigma<\delta-1$, and $\beta\leq\gamma$;

\end{countup}
Then $\S(\K^{<\alpha}, S, \delta)$ holds by defining $\UU_K(\K^{<\alpha},
S, \sigma)=\UU_K(\K^\gamma, S, \sigma)$ for any $\sigma<\delta-1$ and
$K\in\K^\gamma$.
\end{lemma}
\begin{proof}
Let $K\in\K^\gamma$, $\sigma<\delta-1$, and $g\in H^\beta\setminus H^{<\beta}$ where
$\beta, \gamma<\alpha$. Suppose
$\so(g, \laso, \kw(\K^{<\alpha}))=\sigma'\leq\sigma$. Then by
Lemma~\ref{psestack} $\so(g, \laso, \kw(\K^\beta))=\sigma'$. If
$\beta\leq\gamma$ then $g\not\in\UU_K(\K^\gamma,
S, \sigma)=\UU_K(\K^{<\alpha}, S, \sigma)$. If $\beta>\gamma$ then
$g\not\in K$.
\end{proof}

If the new points are sufficiently `far' from a given set, the next
lemma shows that the sequential order of the points `near' the set is
not affected.
{%
\def\alphaDm{\delta_\gamma^m}%
\def\lad{\langle \Dg\rangle}%
\def\Dg{D_\gamma}%
\begin{lemma}\label{w1.kdepth.reduce}
Let $\alpha<\omega_1$ and let $\hku{^\gamma}$ be a primitive
sequential extension of $\hku{^{<\gamma}}$ over some $\Dg$ for every
$\gamma<\alpha$. Let $A\subseteq G=H^{-1}$,
$\hd{\K^{<\gamma}}{\Dg}{A}=\omega_1$, and $\Dg$
satisfy \MM{A, \K^{<\gamma}}.

Then $\so(g, A, \kw(\K^{<\alpha}))=\so(g, A, \kw(\K^{<0}))$.
\end{lemma}
\begin{proof}
Suppose the Lemma holds for all $g'\in H^{<\alpha}$ such
that $\so(g', A, \kw(\K^{<\alpha}))=\sigma'<\sigma$ for some
$\sigma<\omega_1$ and let $\so(g, A, \kw(\K^{<\alpha}))=\sigma$. Let
$g_n\to g$ in $\kw(\K^{<\alpha})$ so that $\so(g_n,
A, \kw(\K^{<\alpha}))<\sigma$. By the hypothesis $\so(g_n,
A, \kw(\K^{<0}))<\sigma$. There exists a $\gamma<\alpha$ such that
$\set g_n:n\in\omega.\subseteq K\in\K^\gamma$ so $g_n\to g$ in
$\kw(\K^\gamma)$ and $\so(g, A, \kw(\K^\gamma))=\sigma$. Let $\gamma$
be the smallest such and suppose $\gamma\geq0$.

Since $\hd{\K^{<\gamma}}{\Dg}{A}=\omega_1$ and $\Dg$
satisfies \MM{A, \K^{<\gamma}}, the sequential order $\so(g,
A, \kw(\K^{<\gamma}))=\sigma$ by Lemma~\ref{D.flat} so there are
$g_n\to g$ in $\kw(\K^{\gamma'})$ for some $\gamma'<\gamma$ and
$\so(g_n, A, \kw(\K^{<\alpha}))=\so(g_n, A, \kw(\K^{<0}))<\sigma$
contradicting the choice of $\gamma$.
\end{proof}
}%

The constructions above are intended for the case when the
sequential order is reflected by a single point in the group, i.e.\
when the sequential order desired is a successor ordinal. In the case
when the group to be constructed must have a limit sequential order,
some additional tools are required.

The following lemma presents a rough idea of how one might go about
handling the limit case. Note that the sequentiality of the
$\Sigma$-product of sequential spaces was proved in~\cite{Is},
Corollary~2.5
(one can also use an argument similar to the one below).

\begin{lemma}\label{sigma.so}
Let $G$ be a $\Sigma$-product of countably compact sequential spaces
$G_\alpha$, $\alpha<\gamma$ for some ordinal $\gamma$. Then
$$
\so(G)\leq\sup\set\min\sum_{\alpha\in
I}\so(G_\alpha):I\subseteq\gamma, |I|<\omega.+1.
$$
Here the $\min$ is taken over all the possible reorderings of $I$.
\end{lemma}
\begin{proof}[Sketch of proof]
Below we refer to the `center' point of each $G_\alpha$, as well as
the corresponding points in $\Sigma G_\alpha$ and finite products of
$G_\alpha$ as $0$.
Let $A\subseteq G$ be such that $\cl{A}\ni0$. Introduce a finer
topology $\tau$ on $G$ by making each $G_\alpha$ discrete. It is well
known that $\tau$ is F\'echet. Consider a finite $I\subseteq\gamma$
and a basic
neighborhod $U_I=\set p\in G
:p(\alpha)=0\mbox{ if }\alpha\in I.$ of $0$ in $\tau$. Consider the natural
projection $\pi_I(A)$ of $A$ into the finite 
product $\prod_{\alpha\in I}G_\alpha$. Note that
$\cl{\pi_I(A)}\ni0$ so by \cite{NS}, Theorem~2.2
$0\in[\pi_I(A)]_{\sigma_I}$ where $\sigma_I\leq\min\sum_{\alpha\in I}\so(G_\alpha)$. Using
induction on $\sigma_I$ (and the sequential compactness of $G$) one
shows that there exists an $x\in U_I$ such that
$x\in[A]_{\sigma_I}$. Thus $\cl{X}^\tau\ni0$ where $X=\set
x\in[A]_{\sigma_I}:I\subseteq\gamma, |I|<\omega.$. Since $\tau$ is
Fr\'echet, this completes the proof.
\end{proof}
The lemma above exhibits two obstacles to obtaining a group of a
sequential order that is a limit ordinal. Consider the case of
$\omega+\omega$. If the product approach suggested by the lemma above
is to be followed, it is clear that infinitely many factors are
required, infinitely may of which must have the sequential order above
$\omega+1$. The best estimate for the sequential order of a product of
$n$ of such factors (obtained in~\cite{NS}) will then exceed
$n\omega>\omega+\omega$. Thus one must be concerned with providing a
better growth rate for the sequential order of finite products of
groups.

The other obstacle is the $+1$ part of the uppper bound estimate. Even
if the sequential order of finite products can be made to grow slowly,
the points in the `last sequence' converging to a given point in the
product might still have unbounded sequential orders thus leading to a
successor
sequential order for the group.

It is unclear to the authors whether the examples below can be given a
product structure, thus a different approach was taken. The examples
`look like' products of groups, although the countable product
structure is destoyed.

The lemma below is a parity neutral version of Lemma~\ref{pparity}.
\begin{lemma}\label{D.e.flat}
Let $\gku$ be a convenient triple.
Let $D\subseteq G$ be such that
$\hd{\K}{D}{\la{A}}=\alpha_\K\leq\omega_1$. Suppose
$D\subseteq\cl{\la{A}}^\kwk$ if $\alpha_\K<\omega_1$, and $D$
satisfies~\MM{\la{A},\K}. Let $\gkup$ be a primitive sequential
extension of $\gku$ over $D$.

If $g\in\overline{\la{A}}^{\kw(\K')}$ then $g=a+s$ where
$a\in\overline{\langle A\rangle}^{\kw(\K)}$ and $s\in
L=\overline{\lad}^{\tau(\U')}$.
\end{lemma}
\begin{proof}
Let $\so(g, \la{A}, \kw(\K'))=\sigma$. If $\sigma\leql\alpha_\K$ by
Lemma~\ref{D.flat} $\so(g, \la{A}, \kwk)=\sigma$. 
Suppose the statement holds for all $g'\in G$ such that
$\so(g', \la{A}, \kw(\K'))=\sigma'<\sigma$ where
$\sigma\geql\alpha_\K$.

Pick $g_n\to g$ in $\kw(\K')$ such that
$\so(g_n, \la{A}, \kw(\K'))=\sigma_n<\sigma$. By the hypothesis
$g_n=g^n+s_n$ where $g^n\in\overline{\la{A}}^\kwk$ and $s_n\in
L$. Since $g_n\to g$ one can find a $K\in\K$, $a_n\in K$, $s^n\in L$
such that $g_n=a_n+s^n$, $a_n\to a$, $s^n\to s$. Then
$d_n=s_n+s^n=g^n+a_n\in G$ so by~\ref{pse.resolve}
$d_n\in\lad$. Thus $a_n=g^n+d_n$ and
$a_n\in\overline{\la{A}}^\kwk$ by Lemma~\ref{set.additivity}. Thus
$g=a+s$ where $a\in\overline{\la{A}}^\kwk$ and $s\in L$.
\end{proof}
The next definition itroduces a product-like structure into the
group. The preservation of this structure is the subject of
several lemmas that follow.
\begin{definition}\label{discrete.separated}
Let $\gku$ be a convenient triple. Let $\CC=\set
S^i:i\in\omega.$ be a family of countable subsets of $G$. Call $\CC$ {\em
discrete separated by $\K$ in $G$\/} if for any infinite, closed and discrete
$D^i=\set d_n^i:n\in\omega.\subseteq\cl{\lasi}^\kwk$, $i\in I$ where $I\subseteq\omega$ is
finite the set $\set\sum_{i\in I}d_n^i:n\in\omega.\not\subseteq K$ for
any $K\in\K$ (equivalently, $\set\sum_{i\in I}d_n^i:n\in\omega.$ is
infinite, closed, and discrete).
\end{definition}

The following lemma follows from the definition of discrete separation.

\begin{lemma}\label{ds.closed.sums}
Let $\gku$ be a convenient triple. Let $\CC=\set S^i:i\in\omega.$ be a
family of countable subsets of $G$ discrete separated by $\K$ in $G$.

Let $I\subseteq\omega$ be finite and let $A_i\subseteq\cl{\lasi}^\kwk$
be closed (in $\kwk$) subsets of $G$ for $i\in I$. Then $\sum_{i\in
I}A_i$ is closed in $\kwk$. 
\end{lemma}

\begin{lemma}\label{ds.stable}
Let $\gku$ be a convenient triple. Let $\CC=\set S^i:i\in\omega.$ be a
family of countable subsets of $G$ discrete separated by $\K$ in $G$.

Let $D\subseteq G$ be a countable closed discrete subspace such that $D$
satisfies property~\MM{\lasi,\K} for every $i\in\omega$, and
$D\subseteq\overline{\lasm}^\kwk$ for some $m\in\omega$. Let $\gkup$
be a primitive sequential extension of $\gkup$ over $D$. Then $\CC$ is
discrete separated by $\K'$ in $G'$.
\end{lemma}
\begin{proof}
Suppose there exist infinite closed and discrete in $\kw(\K')$
subspaces $D^i=\set
d_n^i:n\in\omega.\subseteq\overline{\lasi}^{\kw(\K')}$ where $i\in I$ for
some finite $I\subseteq\omega$, and a
$K'\in\K'$ such that $\set\sum_{i\in I}d_n^i:n\in\omega.\subseteq K'$.

Suppose $i\not=m$ and $\hd{\K}{D}{\lasi}<\omega_1$. Let $D=\set
p_n:n\in\omega.$ and $p_n=a_n'+p_n'$ where $a_n'\in K''$ for some
$K''\in\K$ and $p_n'\in\cl{\lasi}^{\kw(\K)}$. Passing to a subset, if
necessary, we may assume that $p_n'\not=p_k'$ if $n\not=k$, so
$D'=\set p_n':n\in\omega.\subseteq\cl{\lasi}^{\kw(\K)}$ is infinite,
closed, and discrete in $\kw(\K)$. Then $\set
p_n+p_n':n\in\omega.=\set a_n':n\in\omega.\subseteq K''$ contradicting
the property that $\CC$ is discrete separated by $\K$. Hence
$\hd{\K}{D}{\lasi}=\omega_1$ for every $i\not=m$. By Lemma~\ref{D.flat}
$d_n^i\in\cl{\lasi}^\kwk\subseteq G$ for $i\not=m$.

Let $K'\subseteq K+L$ where $L=\cl{\lad}^{\tau(\U')}$ and
$K\in\K$. Now $\sum_{i\in I}d_n^i=a_n+d_n$ where $a_n\in K$ and
$d_n\in L$. If $m\not\in I$ then
$d_n\in G$ by the argument in the preceeding paragraph, so
by~\ref{pse.resolve} and the choice of $D$,
$d_n\in\lad\subseteq\cl{\lasm}^{\kw(\K)}$. By passing to a subset if
necessary, we may assume that the set $D^m=\set d_n=d_n^m:n\in\omega.$ is
either infinite, closed, and discrete, or is a subset of some $K''\in\K$. In the first case
$\set\sum_{i\in I\cup\{m\}}d_n^i:n\in\omega.\subseteq K$, while in the
second case $\set\sum_{i\in I}d_n^i:n\in\omega.\subseteq K+K''\in\K$,
contradicting the property that $\CC$ is discrete separated. Thus
we may assume $m\in I$.

Since $L\subseteq\cl{\lasm}^{\kw(\K')}$, the set $\set
d_n^m+d_n:n\in\omega.\subseteq\cl{\lasm}^{\kw(\K')}$ is also closed and
discrete in $\kw(\K')$ so we may assume, after replacing $D^m$, if
necessary, that $\sum_{i\in
I\setminus\{m\}}d_n^i+d_n^m=a_n\in G$. Then $d_n^m\in G$ and by
Lemma~\ref{D.e.flat} $d_n^m=d_n'+s_n$ where
$d_n'\in\overline{\lasm}^\kwk$ and $s_n\in L$. Since $d_n^m\in G$,
$s_n\in\lad$ by~\ref{pse.resolve} so $D^m\subseteq\cl{\lasm}^\kwk$ by
Lemma~\ref{set.additivity}. Now
$D^i\subseteq\overline{\lasi}^\kwk$ are closed and discrete in $\kwk$
contradicting the assumption that $\CC$ is discrete separated in $G$.
\end{proof}

\begin{lemma}\label{ds.flat}
Let $\gku$ be a convenient triple. Let $\CC=\set
S^i:i\in\omega.$ be a family of countable subsets of $G$ discrete
separated by $\K$ in $G$. Let $D\subseteq G$ be a countable closed
discrete subspace that satisfies~\MM{\CC, \K}.

Let $\gkup$ be a primitive sequential extension of $\gku$ over
$D$. Then $\CC$ is discrete separated by $\K'$ in $G'$.
\end{lemma}
\begin{proof}
Note that $\hd{\K}{D}{\lasi}=\omega_1$ for every $i\in\omega$.

Suppose there exist infinite closed and discrete subspaces $D^i=\set
d_n^i:n\in\omega.\subseteq\overline{\lasi}^{\kw(\K')}$ such that
$\set\sum_{i\in I}d_n^i:n\in\omega.\subseteq K'\in\K'$.
Let $K'\subseteq K+L$ where $L=\overline{\lad}^{\tau(\U')}$ and
$K\in\K$. Now $\sum_{i\in I}d_n^i=a_n+d_n$ where $a_n\in K$ and $d_n\in L$.

By Lemma~\ref{D.flat} $d_n^i\in\overline{\lasi}^\kwk$ for every
$i\in I$, $n\in\omega$ so $d_n\in G$ and by~\ref{pse.resolve}
$d_n\in\lad$. Use \MM{\CC, \K} to find a finite $F\subseteq D$ such
that
$(d+K)\cap(\sum_{i\in I}\overline{\lasi}^\kwk)=\varnothing$
if $d\in\lad\setminus\laf$. Thus $\set d_n:n\in\omega.\subseteq\laf$ and
$\set\sum_{i\in I}d_n^i:n\in\omega.\subseteq K''$ for some $K''\in\K$
contradicting that $\CC$ is discrete separated by $\K$ in $G$.
\end{proof}
It is not possible to preserve the sequential order of every new point
(even if a uniform bound is imposed) in the extension so the next
lemma deals with a much weaker property.

{%
\def\alphaDm{\delta_\gamma^m}%
\def\lad{\langle \Dg\rangle}%
\def\Dg{D_\gamma}%
\begin{lemma}\label{so.reflect}
Let $\alpha<\omega_1$ and let $\hku{^\gamma}$ be a primitive
sequential extension of $\hku{^{<\gamma}}$ over some $\Dg\subseteq
H^{<\gamma}$ for every $\gamma<\alpha$. Let $S^m\subseteq G=H^{<0}$, $m\in\omega$
be countable subsets. Let
$\hd{\K^{<\gamma}}{\Dg}{\lasm}=\alphaDm$. Suppose for every
$m\in\omega$ the set $\Dg$ satisfies~\MM{\lasm, \K^{<\gamma}} and has
the property that either $\alphaDm=\omega_1$ or
$D_\gamma\subseteq\overline{\lasm}^{\kw(\K^{<\gamma})}$.

If $g\in H^\beta\setminus H^{<\beta}$ for some $\beta<\alpha$ and
$g\in\overline{\lasm}^{\kw(\K^{<\alpha})}$ then
$g\in\overline{\lasm}^{\kw(\K^\beta)}$.
\end{lemma}
\begin{proof}
Suppose the statement holds for all $g'\in H^{<\alpha'}$ such that
$\so(g', \lasm, \kw(\K^{<\alpha'})=\sigma'$ where
$(\alpha',\sigma')<(\alpha,\sigma)$. Let
$\so(g, \lasm, \kw(\K^{<\alpha}))=\sigma$. If $\sigma=0$ then
$g\in\lasm$, otherwise there are $g_n\to g$ in $\kw(\K^{<\alpha})$
such that $\so(g_n, \lasm, \kw(\K^{<\alpha}))=\sigma_n<\sigma$. Since $g_n\to
g$ there exists a $K\in\K^\gamma$ for some $\gamma<\alpha$ such that
$g_n\in K$.

Suppose $\alpha$ is a successor ordinal. If $\beta=\alpha-1$
then $\K^{<\alpha}=\K^\beta$ so the statement holds. Otherwise $\beta<\alpha-1$ and $g\in
H^{<\alpha-1}$ so by Lemma~\ref{D.e.flat} $g=a+s$ where
$a\in\cl{\lasm}^{\kw(\K^{<\alpha-1})}$ and
$s\in\cl{\la{D_{\alpha-1}}}^{\tau(\U^{\alpha-1})}\subseteq\cl{\lasm}^{\kw(\K^{\alpha-1})}$
if $\delta_{\alpha-1}^m<\omega_1$. If $\delta_{\alpha-1}^m=\omega_1$
then $g\in\cl{\lasm}^{\kw(\K^{<\alpha-1})}$ by
Lemma~\ref{D.flat}. Since $a,g\in H^{<\alpha-1}$, it follows that
$s\in H^{<\alpha-1}$, so by~\ref{pse.resolve} $s\in\la{D_{\alpha-1}}$ and
$g\in\overline{\lasm}^{\kw(\K^{<\alpha-1})}$. Since $(\alpha-1,\sigma)<(\alpha,\sigma)$,
the statement holds by the hypothesis.

If $\alpha$ is a limit ordinal $\gamma<\alpha'$ for some
$\alpha'<\alpha$. Since $(\alpha,\sigma_n)<(\alpha,\sigma)$ and
$g_n\in H^\gamma$, by the hypothesis
$g_n\in\overline{\lasm}^{\kw(\K^{<\alpha'})}$ so
$g\in\overline{\lasm}^{\kw(\K^{<\alpha'})}$ and
$g\in\overline{\lasm}^{\kw(\K^\beta)}$ by the hypothesis.
\end{proof}
}%

{%
\def\alphaDm{\delta_\gamma^m}%
\def\lad{\langle \Dg\rangle}%
\def\Dg{D_\gamma}%
\begin{lemma}\label{ds.limit}
Let $\alpha<\omega_1$ and let $\hku{^\gamma}$ be a primitive
sequential extension of $\hku{^{<\gamma}}$ over some $\Dg$ for every
$\gamma<\alpha$. Let $S^i\subseteq G=H^{<0}$ be countable subsets such
that $\CC=\set S^i:i\in\omega.$ is discrete separated by $\K^\gamma$
in $H^\gamma$ for every $\gamma<\alpha$. Let
$\hd{\K^{<\gamma}}{\Dg}{\lasm}=\alphaDm$. Suppose for every
$m\in\omega$ the set $\Dg$ satisfies~\MM{\lasm, \K^{<\gamma}} and has
the property that 
whenever $\alphaDm<\omega_1$ the set
$\Dg\subseteq\overline{\lasm}^{\kw(\K^{<\gamma})}$.

Then $\CC$ is discrete separated by $\K^{<\alpha}$ in $H^{<\alpha}$.
\end{lemma}
\begin{proof}
Suppose the statement holds for all $\gamma<\alpha$.
Let $d_n^i\in H^{\alpha_n^i}\setminus H^{<\alpha_n^i}$ and
$K\in\K^\beta$, $\beta<\alpha$ be such that $D=\set\sum_{i\in
I}d_n^i:n\in\omega.\subseteq K$, where $D^i=\set
d_n^i:n\in\omega.\subseteq\overline{\lasi}^{\kw(\K^{<\alpha})}$ are
closed discrete subspaces of $H^{<\alpha}$, $i\in
I$, and $I\subseteq\omega$ is finite.

If $\gamma=\max\{\set\sup\alpha_n^i:i\in I.\cup\{\beta\}\}<\alpha$ then
by Lemma~\ref{so.reflect}
$D^i\subseteq\overline{\lasi}^{\kw(\K^\gamma)}$ so $\CC$ is not
discrete separated in $\kw(\K^\gamma)$.

Suppose, say, $\alpha_n^0$ is unbounded in $\alpha$. Pick $n\in\omega$
so that $\alpha_n^0>\beta$ and let $\gamma=\max_{i\in
I}\alpha_n^i$. By the inductive hypothesis, $\CC$ is discrete
separated by $\K^{<\gamma}$ in $H^{<\gamma}$. Since $\la{D_\gamma}$ is
closed and discrete, this implies that
$\alphaDm<\omega_1$ for at most one $m\in\omega$ (note that $D_\gamma$
cannot be finite, otherwise $H^\gamma\setminus
H^{<\gamma}=\varnothing$). Then $\gamma>\beta$
and
$\hd{\K^{<\gamma}}{\Dg}{\lasi}=\omega_1$ for all $i\in I$ such that
$i\not=m$ for some $m\in\omega$. By Lemma~\ref{D.flat}
$d_n^i\in\overline{\lasi}^{\K^{<\gamma}}$ for every $i\not=m$ so
$m\in I$ (otheriwse $\max_{i\in
I}\alpha_n^i<\gamma$). Thus $\sum_{i\in I}d_n^i\in H^\gamma\setminus H^{<\gamma}$
contradicting $\sum_{i\in I}d_n^i\in K\subseteq H^{<\gamma}$.
\end{proof}
}%
Putting the concepts introduced above together, the next definition
introduces a basic `extension step' of the construction.

\begin{definition}\label{fine.pse}
Let $\gku$ and $\gkup$ be convenient triples, $\CC=\set
S^i:i\in\omega.$ be a family of countable subsets of $G$, and
$\set\sigma^i:i\in\omega.\subseteq\omega_1$ be a set of successor
ordinals. Let $D\subseteq G$ be a countable independent closed and
discrete subset of $G$ in $\kwk$. Call $\gkup$ a {\em fine primitive
sequential extension (or fpse for short) of $\gku$ over $(D, \CC)$ (or
$D$ if $\CC$ is clear)\/} if the following properties hold.
\begin{countup}[bcp]
\item\label{fine.pse.pse}
$\gkup$ is a primitive sequential extension of $\gku$ over $D$, which
satisfies \MM{\K, \lasi} for every $i\in\omega$, and
either $D$ satisfies \MM{\K, \CC} or $D\subseteq\cl{\lasm}^\kwk$ for
some $m\in\omega$;

\item\label{fine.pse.sep}
if $\hd{\K}{D}{\lasom}=\delta_D^m<\sigma^m-1$ then
$D\subseteq[\lasom]_{\delta_D^m}^\kwk$, $D$ satisfies \MM{\K, \lasom},
and
$\overline{\lado}^{\tau(\U')}\cap \overline{\lade}^{\tau(\U')}=\varnothing$;

\end{countup}
\end{definition}

\begin{definition}\label{fpse.chain}
Let $\gku$ be a convenient triple. Let $\CC=\set S^i:i\in\omega.$ be a
family of countable subsets of $G$. Call a family
$C=\set\hku{^\gamma}:\gamma^0\leq\gamma<\alpha.$ of convenient triples an {\em
fpse-chain above $\gku$ along $\set D_\gamma:\gamma<\alpha.$ relative
to $\CC$\/} if $\hku{^\gamma}$ is an fpse of $\hku{^{<\gamma}}$ over
$(D_\gamma,\CC)$ where $D_\gamma\subseteq H^{<\gamma}$ for $\gamma>\gamma^0$
and $\hku{^{\gamma^0}}$ is an fpse of $\gku$ over $(D_{\gamma^0}, \CC)$ where
$D_{\gamma^0}\subseteq G$. Call each $D_\gamma$ the {\em free sequence at
$\gamma$}.
\end{definition}

{%
\def\alphaDm{\delta_\gamma^m}%
\def\lad{\langle\Dg\rangle}%
\def\Dg{D_{\gamma'}}%
\begin{lemma}\label{w1.kdepth.sum}
Let $\alpha<\omega_1$ and let $\hku{^\gamma}$ be a primitive
sequential extension of $\hku{^{<\gamma}}$ over some
$D_\gamma\subseteq H^{<\gamma}$ for every
$\gamma^0<\gamma<\alpha$. Let $S^i\subseteq H^{\gamma^0}$, $i\in\omega$ be
such that for every $\gamma<\alpha$ there is at most one $i\in\omega$ such that
$\hd{\K^{<\gamma}}{D_\gamma}{\lasi}<\omega_1$, $D_\gamma$
satisfies \MM{\lasi, \K^{<\gamma}} for every $\gamma^0<\gamma<\alpha$
and every $i\in\omega$, and
$D_\gamma\subseteq\cl{\lasi}^{\kw(\K^{<\gamma})}$ whenever
$\hd{\K^{<\gamma}}{D_\gamma}{\lasi}<\omega_1$.

Let $g=\sum_{i\in I}g_i$ where
$g_i\in\cl{\lasi}^{\kw(\K^{<\alpha})}$ for some finite
$I\subseteq\omega$. If $g\in H^\gamma\setminus H^{<\gamma}$ and
$g_i\in H^{\gamma_i}\setminus H^{<\gamma_i}$ then
$g_i\in\cl{\lasi}^{\kw(\K^{\gamma_i})}$ and $\gamma=\max_{i\in
I}\gamma_i$.
\end{lemma}
\begin{proof}
By Lemma~\ref{so.reflect} $g_i\in\overline{\lasi}^{\kw(\K^{\gamma_i})}$. 
Let $\gamma'=\max_{i\in I}\gamma_i$. If $\gamma'=\gamma^0$ then
$\gamma_i=\gamma^0$ for every $i\in I$ so $g\in H^{\gamma^0}$. Otherwise
$\hku{^{\gamma'}}$ is a primitive sequential extension of
$\hku{^{<\gamma'}}$ over $\Dg\subseteq H^{<\gamma'}$. If
$\hd{\K^{<\gamma'}}{\Dg}{\lasi}=\omega_1$ then
$g_i\in\cl{\lasi}^{\kw(\K^{<\gamma'})}$ by Lemma~\ref{D.flat} so
$\gamma_i<\gamma'$. Thus there is a unique $j\in I$ such that
$\gamma_j=\gamma'$ and $\gamma_i<\gamma'$ for $i\not=j$. Now
$g=\sum_{i\in I}g_i$ so $\gamma'=\gamma$.
\end{proof}
}%
Certain discrete separated families behave like direct sums
algebraically.

\begin{lemma}\label{pse.sum}
Let $\gkup$ be a pse of $\gku$ over
$D\subseteq G$. Let $S^i\subseteq G$, $i\in\omega$ be such that
$\cl{\lasm}^\kwk\cap(\sum_{i\in I}\cl{\lasi}^\kwk)=\{0\}$ for any
finite $I\subseteq\omega\setminus\{m\}$. Suppose there is at
most one $i\in\omega$ such that $\hd{\K}{D}{\lasi}<\omega_1$ and if
such $i$ exists $D\subseteq\cl{\lasi}^\kwk$. Let $D$
satisfy \MM{\K, \lasi} for every $i\in\omega$.

Then $\cl{\lasm}^{\kw(\K')}\cap(\sum_{i\in
I}\cl{\lasi}^{\kw(\K')})=\{0\}$ for any finite
$I\subseteq\omega\setminus\{m\}$.
\end{lemma}
\begin{proof}
Let $g=\sum_{i\in I}g_i$ where $I\subseteq\omega$ is finite,
$g_i\in\cl{\lasi}^{\kw(\K')}$ for $i\in I$, $m\not\in I$, and
$g\in\cl{\lasm}^{\kw(\K')}$.

Suppose $\hd{\K}{D}{\lasm}=\omega_1$. Then $g\in\overline{\lasm}^\kwk$
by Lemma~\ref{D.flat} and $g_i\in\overline{\lasi}^\kwk$ by
Lemma~\ref{w1.kdepth.sum} so $g=0$ by the property of $\CC$.

If $\hd{\K}{D}{\lasm}<\omega_1$ and $g\in G'\setminus G$ then
$\hd{\K}{D}{\lasi}=\omega_1$ and $g_i\in\cl{\lasi}^\kwk$ for $i\in
I$. Thus $g=\sum_{i\in I}g_i\in G$ contradicting the assumption on
$g$. Thus $g\in G$ and by Lemma~\ref{D.e.flat} $g=a+s$ where
$a\in\cl{\lasm}^\kwk$ and $s\in\cl{D}^{\tau(\U')}$. Since $g\in G$,
$s\in G$ and by~\ref{pse.resolve} $s\in\lad$. Thus
$g\in\cl{\lasm}^\kwk$ and $g=0$ by the condition on $\CC$.
\end{proof}
It is convenient to have an upper bound measure
for the sequential order of
the construction.

\begin{definition}\label{seq.scale}
Let $\gku$ be a convenient triple. Let $\CC=\set S^i:i\in\omega.$ be a
family of countable subsets of $G$ and
$\xi=\set\sigma^i:i\in\omega.\subseteq\omega_1$ be a set of successor
ordinals. Call $(\CC, \xi)$ a {\em sequential scale in $\gku$\/} if
the following conditions are satisfied.
\begin{countup}[bcp]
\item\label{seq.scale.ds}
$\CC$ is discrete separated by $\K$ in $G$;

\item\label{seq.scale.direct}
$\cl{\lasm}^\kwk\cap\sum_{i\in I}\cl{\lasi}^\kwk=\{0\}$ for every
finite $I\subseteq\omega\setminus\{m\}$;

\item\label{seq.scale.so}
$\S(\K, S^i, \sigma^i)$ holds for every $i\in\omega$;

\end{countup}
\end{definition}
A trivial observation shows that a collection $\CC=\set
S^i:i\in\omega.$ of subsets will form a sequential scale for {\em any\/}
$\xi\subseteq\omega_1$ provided $\cup_{i\in\omega}S^i$ is independent,
and $\la{\cup_{i\in\omega}S^i}$ is closed
and discrete in $\kwk$.

\begin{lemma}\label{fpse.successor}
Let $\gku$ be a convenient triple, $\CC=\set S^i:i\in\omega.$ be such
that $(\CC,\xi)$ is a sequential scale in $\gku$ for some $\xi\subseteq\omega_1$. Let
$P\subseteq G$ be an infinite subset closed and discrete in
$\kwk$. Suppose $P\subseteq\cl{\lasm}^\kwk$ or $\hd{\K}{P}{\CC}=\omega$.

Then there exists an fpse $\gkup$ of $\gku$ such that $S\to s$ for
some infinite $S\subseteq P$ and $s\in G'\setminus G$.
\end{lemma}
\begin{proof}
Using the observation immediately following Lemma~\ref{kdshift} and
passing to a subset if neccessary,
assume that $\hd{\K}{P'}{\lasom}=\sigma$ for every infinite
$P'\subseteq P$. Let $P=\set p_n:n\in\omega.$,
$D=\set d_n:n\in\omega.\subseteq[\lasom]_\sigma^\kwk$ if
$\sigma<\sigma^m-1$, $K\in\K$, and $a_n, a\in K$ be such that
$p_n=d_n+a_n$ and $a_n\to a$. If $\sigma\geq\sigma^m-1$ put
$D=P$. Using Lemma~\ref{free.sequence} find an infinite 
$J\subseteq\omega$ such that the set $Q=\set d_n:n\in J.$
satisfies \MM{\lasi, \K} for $i\in\omega$ and
satisfies \MM{\lasom, \K} if $\sigma<\sigma^m-1$. Note that
$\hd{\K}{Q}{\lasom}=\sigma$ by the assumption on $P$ and
Lemma~\ref{kdshift}.

Use Lemma~\ref{ct.split} to find a family of clopen subgroups
$\U_0\supseteq\U$ of finite index in $G$ such that
$\cl{\la{Q'}}^{\tau(\U_0)}\cap G=\la{Q'}$ for every infinite $Q'\subseteq Q$
and
$\cl{\le{Q}}^{\tau(\U_0)}\cap\cl{\lo{Q}}^{\tau(\U_0)}=\varnothing$ if
$\sigma<\sigma^m-1$ (all closures with respect to $\tau(\U_0)$ are
assumed to be taken in the appropriate compact completion of $G$). By
passing to a subset if necessary assume $Q\to
q$ in $\tau(\U_0)$ for some $q$ in the compact completion of $G$. Put
$L=\cl{\la{Q}}^{\tau(\U_0)}$, let $G'=G+L$, and let $\K'$ be the closure
of $\K\cup\{L\}$ under finite sums and intersections. Let $\U'$ be the
trace of $\U_0$ on
$G'$. Now~\ref{pse.order}, \ref{pse.resolve}, \ref{fine.pse.pse},
and~\ref{fine.pse.sep}
follow from the construction and $S=\set d_n+a_n:n\in J.\to q+a=s$ in
$\tau(\U')$ so $\gkup$ is the desired fpse.

If $\hd{\K}{P}{\CC}=\omega$ use Lemma~\ref{free.sequence} to find an infinite
$S\subseteq P$ that satisfies \MM{\CC, \K} and use an argument similar
to the one above to construct $\gkup$.
\end{proof}

\begin{lemma}\label{ssc.resolve}
Let $\gku$ be a convenient triple, $\CC=\set S^i:i\in\omega.$ be such
that $(\CC,\xi)$ is a sequential scale in $\gku$ for some $\xi$. Let
$I\subseteq\omega$ be finite, and let $U\in\U$.

Then there exists a $U'\subseteq U$ such that $(\sum_{i\in
I}\cl{\lasi}^\kwk)\cap U'=\sum_{i\in I}(\cl{\lasi}^\kwk\cap U)$ and
$U'$ is a clopen subgroup of finite index in $(G,\kw(\K))$.
\end{lemma}
\begin{proof}
The group $H=\sum_{i\in I}(\cl{\lasi}^\kwk\cap U)$ is closed in $\kwk$
by Lemma~\ref{ds.closed.sums}
and is of finite index in $\sum_{i\in I}\cl{\lasi}^\kwk$. Let
$g_1,\ldots,g_k\in G\setminus H$ be such that $H+\set g_i:i\leq
k.\supseteq\sum_{i\in I}\cl{\lasi}^\kwk$. Use Lemma~\ref{bhb} to construct an open
subgroup of finite index $U_i$, $i=1,\ldots,k$ such that $H\subseteq
U_i$ and $g_i\not\in U_i$. Put $U'=\cap_{i\leq k}U_i\cap U$.
\end{proof}
To keep the sequential order low one must be able to construct
sequences converging to some points that are already present in the
group.

\begin{lemma}\label{z.group}
Let $\gku$ be a convenient triple, $\CC=\set S^i:i\in\omega.$ and
$\xi=\set\sigma^i:i\in\omega.\subseteq\omega_1$ be such that
$(\CC,\xi)$ is a sequential scale in $\gku$. Let $I\subseteq\omega$ be
finite, $D^i=\set d_n^i:n\in\omega.\subseteq\cl{\lasi}^\kwk$ be such
that $D^i\to0$ in $\tau(\U)$ for every $i\in I$, and
$\hd{\K}{D^{i'}}{\lo{S^{i'}}}\geq\sigma^{i'}-1>0$ for some $i'\in I$.

Then there exists an infinite $J\subseteq\omega$ and an fpse $\gkup$ of
$\gku$ over $D=\set d_n^{i'}:n\in J.$ such that $D\to0$ in $\kw(\K')$
and $\set d_n^i:n\in J.\to0$ in $\tau(\U')$ for every $i\in I$.
\end{lemma}
\begin{proof}
Pick an infinite $J\subseteq\omega$ such that either $D(i)=\set
d_n^i:n\in J.\to0$ or $\la{D(i)}$ is closed and discrete in
$\kwk$ and $D(i')$ satisfies \MM{\lasi,\K} for every
$i\in\omega$. Since $\hd{\K}{D^{i'}}{\lo{S^{i'}}}>0$ the set
$\la{D(i')}$ is infinite, closed and discrete in $\kwk$. By taking a subset of
$J$ if necessary assume that $D(i')$ satisfies \MM{\lasoi,\K} and the
set $\cup_{i\in I}D(i)$ is independent (using~\ref{seq.scale.direct}). Since
$\CC$ is discrete separated and $D(i)\subseteq\cl{\lasi}^\kwk$, the
group $\la{\cup_{i\in I'}D(i)}$ is closed and discrete in $\kwk$ where
$I'=\set i\in I:\la{D(i)}\hbox{ is closed discrete in }\kwk.$.

Use Lemma~\ref{bhb} and the argument in Lemma~\ref{ct.resolve} to find
a $\U'\supseteq\U$ such that for every $U\in\U'$ and every $i\in I$
the set $D(i)\setminus U$ is finite and $\cl{\la{\cup_{i\in
I'}D(i)\setminus F_i}}^{\tau(\U')}\cap G=\la{\cup_{i\in
I'}D(i)\setminus F_i}$ for any finite $F_i\subseteq G$. Then
$D(i)\to0$ in $\tau(\U')$.

Put $L=\cl{D}^{\tau(\U')}$ and let $\K'$ be the closure of
$\K\cup\{L\}$ under finite intersections and
sums. Now $D(i')\to0$ in $\kw(\K')$, \ref{pse.order}
and~\ref{pse.resolve} hold, and, since 
$\hd{\K}{D}{\lasoi}\geq\sigma^{i'}-1$,
properties~\ref{fine.pse.pse}--\ref{fine.pse.sep} hold so $\gkup$
is an fpse of $\gku$ over $D$.
\end{proof}

The next lemma presents the basic structure of the example.
\begin{lemma}\label{fpsestack}
Let $\gku$ be a convenient triple, $\CC=\set S^i:i\in\omega.$ and
$\xi=\set\sigma^i:i\in\omega.$ be such that $(\CC,\xi)$ is a
sequential scale in $\gku$. Let $\set\hku{^\gamma}:\gamma<\alpha.$ be
an fpse-chain over $\gku$ along $\set D_\gamma:\gamma<\alpha.$
relative to $\CC$.

Then $(\CC,\xi)$ is a sequential scale in $\hku{^{<\alpha}}$ and the
following properties hold.
\begin{countup}[bcp]
\item\label{fc.so}
$\UU_K(\K^{<\alpha}, S^i, \sigma)=\UU_K(\K^{\gamma'},
S^i, \sigma)=\UU_K(\K^\gamma, S^i, \sigma)$ for every $K\in\K^\gamma$,
$\gamma\leq\gamma'<\alpha$, and $\sigma<\sigma^i-1$;

\item\label{fc.cl}
if $g\in H^\gamma\setminus H^{<\gamma}$ for $\gamma<\alpha$ and
$g\in\cl{\lasi}^{\kw(\K^{<\alpha})}$ then
$g\in\cl{\lasi}^{\kw(\K^\gamma)}$;

\item\label{fc.scl}
if $g\in H^\gamma\setminus H^{<\gamma}$ for some $\gamma<\alpha$ and
$\so(g, \lasoi, \kw(\K^{<\alpha}))=\sigma<\sigma^i$ then
$\so(g, \lasoi, \kw(\K^\gamma))=\sigma$;

\end{countup}
\end{lemma}
\begin{proof}
Suppose the statement has been proved for all
$\alpha'<\alpha$. Suppose $\alpha$ is a successor ordinal such that
$\alpha'=\alpha-1$. Then by the hypothesis $(\CC,\xi)$ is a sequential
scale in $\hku{^{<\alpha'}}$ and $\hku{^{<\alpha}}=\hku{^{\alpha'}}$ is an fpse of
$\hku{^{<\alpha'}}$ over some $D_{\alpha'}\subseteq
H^{<\alpha'}$. Now~\ref{seq.scale.ds} and~\ref{fc.cl} hold by
Lemmas~\ref{ds.stable}, \ref{ds.flat}, and~\ref{so.reflect},
and~\ref{fine.pse.pse}. Property~\ref{seq.scale.direct} follows from
Lemma~\ref{pse.sum}
and~\ref{fine.pse.pse}. Property~\ref{seq.scale.so} follows from
Lemma~\ref{D.bind} applied to every $S^i$ and
property~\ref{fine.pse.sep}. Moreover, by Lemma~\ref{D.bind} one may
choose $\UU_K(\K^{<\alpha}, S^i, \sigma)=\UU_K(\K^{\alpha'},
S^i, \sigma)=\UU_K(\K^\gamma, S^i, \sigma)$ for any $K\in\K^\gamma$,
$\sigma<\sigma^i-1$, $i\in\omega$, $\gamma<\alpha$ so~\ref{fc.so}
holds. Now~\ref{fc.scl} follows from
Lemma~\ref{psestack}, \ref{fine.pse.pse}, and~\ref{fine.pse.sep}.

Let $\alpha$ be a limit ordinal. Then~\ref{seq.scale.ds} holds by
Lemma~\ref{ds.limit}, the hypothesis,
and~\ref{fine.pse.pse}. Property~\ref{fc.cl} holds by
Lemma~\ref{so.reflect} and~\ref{fine.pse.pse}.

Let $I\subseteq\omega$ be finite and
$g_i\in\cl{\lasi}^{\kw(\K^{<\alpha})}$ for $i\in I$, $m\not\in I$ and
$g_m\in\cl{\lasm}^{\kw(\K^{<\alpha})}$. Then
$g_i\in\cl{\lasi}^{\kw(\K^\gamma)}$ where $\gamma=\max_{i\in
I\cup\{m\}}\gamma_i$ and $g_i\in H^{\gamma_i}\setminus H^{<\gamma_i}$
for $i\in I\cup\{m\}$. Then $\gamma<\alpha$ and
property~\ref{seq.scale.direct} holds by the
hypothesis. Properties~\ref{seq.scale.so} and~\ref{fc.so} hold by
Lemma~\ref{wparity}, \ref{fine.pse.sep}, and the
hypothesis. Property~\ref{fc.scl} holds by Lemma~\ref{psestack}.
\end{proof}

\section{Stability}
A number of arguments below require that new sequences are added in a
strict order (both to make the recursion work, as well as to
keep the convergence structure intact) depending on the sequential
order of the points involved.

The sequential order may change as the new sequences are added, however, so a
mechanism to keep the order fixed is needed. One such mechanism is
described below.

While the results below hold for any type of fpse-chains, the proofs
are more involved and the additional generality is not required in the
constructions that follow. We thus introduce a narrow subclass of
fpse-chains.

\begin{definition}\label{finite.type}
Let $\gku$ be a convenient triple and $\CC=\set S^i:i\in\omega.$ be
such that $(\CC,\xi)$ is a sequential scale in $\gku$ for some
$\xi$. Call an fpse-chain
$C=\set\hku{^\gamma}:\gamma^0\leq\gamma\leq\gamma^1.$ above $\gku$ along
$\set D_\gamma:\gamma^0\leq\gamma\leq\gamma^1.$ relative to
$(\CC,\xi)$ {\em of finite type\/} if
$D_\gamma\subseteq\cl{\la{S^{i(\gamma)}}}^{\kw(\K^{<\gamma})}$ for
every $\gamma$.

If $i(\gamma)=m$ for every $\gamma$ say that $C$ is {\em close to
$\lasm$}. Otherwise, if $i(\gamma)\not=m$ for every $\gamma$ say that
$C$ is {\em away from $\lasm$}.
\end{definition}
The constructions below add new sequences to several `summands' of the
`direct sum'. The next construction presents a way to rearrange the
order in which new sequences are added.

\begin{definition}\label{m.reduce}
Let $\gku$ be a convenient triple and $\CC=\set S^i:i\in\omega.$ be
such that $(\CC,\xi)$ is a sequential scale in $\gku$ for some
$\xi$. Let $C=\set\hku{^\gamma}:\gamma^0\leq\gamma\leq\gamma^1.$ be an
fpse-chain of finite type above $\gku$ along $\set
D_\gamma:\gamma^0\leq\gamma\leq\gamma^1.$ relative to $(\CC,\xi)$.

Let $m\in\omega$ and
$\gamma^{(m)}=\set\gamma^0\leq\gamma\leq\gamma^1:
D_\gamma\subseteq\cl{\lasm}^{\kw(\K^{<\gamma})}.$.
Let $\mu<\omega_1$ be such that $h:\mu\to\gamma^{(m)}$ is the unique
1-to-1 monotone map. Define a chain
$C^{(m)}(\mu')=\set\hku{^\lambda_{(m)}}:\lambda<\mu'\leq\mu.$ above $\gku$ along $\set
D_{h(\lambda)}:\lambda<\mu'.$ recursively as follows.

Let $C^{(m)}(\varnothing)=\varnothing$. Suppose $C^{(m)}(\mu'')$ has
been defined for all $\mu''<\mu'$. If $\mu'$ is a limit ordinal, put
$C^{(m)}(\mu')=\cup_{\mu''<\mu'}C^{(m)}(\mu'')$. Otherwise
$\mu'=\mu''+1$ for some $\mu''<\omega_1$. Put
$L_{\mu'}=\cl{D_{h(\mu')}}^{\tau(\U_{h(\mu')})}$. Let
$H^{\mu'}_{(m)}=H^{\mu''}_{(m)}+L_{\mu'}$ and let $\K^{\mu'}_{(m)}$ be
the closure of $\K^{\mu''}_{(m)}\cup\{L_{\mu'}\}$ under finite sums
and intersections. Let $\U^{\mu'}_{(m)}=\set U\cap
H^{\mu'}_{(m)}:U\in\U^{h(\mu')}.$. In the definition above,
$H^\lambda_{(m)}\subseteq H^{\gamma^1}$ and
$\K^\lambda_{(m)}\subseteq\K^{\gamma^1}$. It is an easy argument to see that
each $\hku{^\lambda_{(m)}}$ is a convenient triple.

Call $C^{(m)}=C^{(m)}(\mu)$ the {\em $m$-reduction of $C$\/}. 
\end{definition}

It is not immediately clear that the $m$-reduction is an
fpse-chain. This is the subject of Lemma~\ref{m.reduce.fpse}.

The lemma below holds for any pse-chains, thus the sequential scale is
only needed to narrow its statement to fpse-chains and plays no other
role in the proof.

\begin{lemma}\label{ct.compacts}
Let $\gku$ be a convenient triple and $\CC=\set S^i:i\in\omega.$ be
such that $(\CC,\xi)$ is a sequential scale in $\gku$ for some
$\xi$. Let $\set\hku{^\gamma}:\gamma\leq\gamma^0.$ be an
fpse-chain above $\gku$ along $\set
D_\gamma:\gamma\leq\gamma^0.$ relative to $(\CC,\xi)$. Let
$L_\gamma=\cl{D_\gamma}^{\tau(\U_\gamma)}$. 

Let $K\subseteq H^\gamma$ be compact in $\kw(\K^\gamma)$. Then
$K\subseteq K'+\sum_{\lambda\in F}L_\lambda$ where $K'\in\K$ and
$F\subseteq\gamma+1$ is finite.
\end{lemma}
\begin{proof} Induction on $\gamma^0$.
\end{proof}
Further similarity with the direct sum is provided by the following
`projection' result.

\begin{lemma}\label{compact.project}
Let $\gku$ be a convenient triple and $\CC=\set S^i:i\in\omega.$ be
such that $(\CC,\xi)$ is a sequential scale in $\gku$ for some
$\xi$. Let $K\subseteq G$ be compact in $\kw(\K)$, $I\subseteq\omega$
be finite and $m\in\omega$. Then the set 
$\pi_{I,m}(K)=\set d\in\cl{\lasm}^\kwk:d+d'\in K\hbox{ for some
}d'\in\sum_{i\in I\setminus\{m\}}\cl{\lasi}^\kwk.$ is compact in $\kwk$.
\end{lemma}
\begin{proof}
Suppose $\pi_{I,m}(K)$ is not compact. Then there exists an infinite
closed and discrete subset $D^m=\set d_n^m:n\in\omega.$ such that for
some $d_n^i\in\cl{\lasi}^\kwk$ and finite $I\subseteq\omega$, $I\ni
m$, and $D=\set\sum_{i\in
I}d_n^i:n\in\omega.\subseteq K$. After thinning out and reindexing we
may assume that each $D^i=\set d_n^i:n\in\omega.$ is either closed and
discrete in $\kwk$ or $D^i\to d^i$ for some $d^i\in G$. Let
$I'\subseteq I$ be the set of all $i\in I$ such that $D^i$ is closed
and discrete in $\kwk$. Then $m\in I'$ and $\set\sum_{i\in 
I'}d_n^i:n\in\omega.\subseteq K'$ for some $K'\in\K$ contradicting the
assumption that $\CC$ is discrete separated.
\end{proof}

\begin{lemma}\label{m.reduce.so}
Let $\gku$ be a convenient triple and $\CC=\set S^i:i\in\omega.$ be
such that $(\CC,\xi)$ is a sequential scale in $\gku$ for some
$\xi$. Let $C=\set\hku{^\gamma}:\gamma^0\leq\gamma\leq\gamma^1.$ be an
fpse-chain of finite type above $\gku$ along $\set
D_\gamma:\gamma^0\leq\gamma\leq\gamma^1.$ relative to $(\CC,\xi)$.

Let $m\in\omega$ and
$\gamma^{(m)}=\set\gamma^0\leq\gamma\leq\gamma^1:
D_\gamma\subseteq\cl{\lasm}^{\kw(\K^{<\gamma})}.$.
Let $\mu<\omega_1$ be such that $h:\mu\to\gamma^{(m)}$ is the unique
1-to-1 monotone map. Let
$C^{(m)}=\set\hku{^\lambda_{(m)}}:\lambda<\mu.$ be the
$m$-reduction of $C$ and let $\so(g,
D, \kw(\K^\gamma))=\sigma<\omega_1$ where $D\subseteq\lasm$.

Then $g\in H^\lambda_{(m)}$ and $\so(g, D, \kw(\K^\lambda_{(m)}))=\sigma$ where
$\lambda=\sup\set\lambda':h(\lambda')\leq\gamma.$. If $\lambda$ is a
limit ordinal and $\gamma\not=h(\lambda)$, $\lambda$ may be replaced with
with $<\lambda$.
\end{lemma}
\begin{proof}
Suppose the satement holds for all $g'\in G^\gamma$ such that $\so(g',
D, \kw(\K^\gamma))=\sigma'<\sigma$ and let $\so(g,
D, \kw(\K^\gamma))=\sigma$. If $\sigma=0$ then the proof is immediate
so assume that $\sigma>0$ and there are $g_n\to g$ in $\kw(\K^\gamma)$
such that $\so(g_n, D, \kw(\K^\gamma))=\sigma_n<\sigma$.

By the hypothesis $\so(g_n, D, \kw(\K^\lambda_{(m)}))=\sigma_n$ where
$\lambda$ is defined as in the statement of the Lemma. Since $g_n\to
g$ in $\kw(\K^\gamma)$ there exists a compact $K\in\K^\gamma$ such
that $g_n\in K$ and by Lemma~\ref{ct.compacts} $K\in K'+\sum_{\nu\in
F}L_\nu$ for some finite $\nu\subseteq\gamma+1$. Thus
$g_n=g^n+\sum_{\nu\in F}g_n^\nu$ where $g^n\in K'\subseteq G$, $g_n^\nu\in L_\nu$, and
$L_\nu=\cl{D_\nu}^{\tau(\U^\nu)}\subseteq\cl{\la{S^{i(\nu)}}}^{\kw(\K^\nu)}$.

Let $I=\set i(\nu):\nu\in F.\cup\{m\}$. By Lemma~\ref{set.additivity}
$g^n\in\sum_{i\in I}\cl{\lasi}^{\kw(\K^\gamma)}$ so $g^n=\sum_{i\in
I}g^n(i)$ where $g^n(i)\in\cl{\lasi}^{\kw(\K^\gamma)}$. Since $g^n\in G$ by
Lemma~\ref{w1.kdepth.sum} $g^n(i)\in\cl{\lasi}^\kwk$ for every $i\in I$.

Since $(\CC,\xi)$ is a sequential scale
$g^n(i)+\sum_{i(\nu)=i}g_n^\nu=0$ for every $i\not=m$ by~\ref{seq.scale.direct} so
$g_n=g^n(m)+g_n^\nu$ where $g^n(m)\in\pi_{I,m}(K')$ and
$g_n^\nu\in\sum_{i(\nu)=m}L_\nu$. If $i(\nu)=m$ then
$L_\nu\in\K^\nu_{(m)}$ where $\nu\leq\gamma$ so $\nu=h(\lambda')$ for
some $\lambda'\leq\lambda$. Thus $g_n\in K''\in\K^\lambda_{(m)}$ and
$\so(g, D, \kw(\K^\lambda_{(m)}))=\sigma$.
\end{proof}

\begin{lemma}\label{m.reduce.hd}
Let $\gku$ be a convenient triple and $\CC=\set S^i:i\in\omega.$ be
such that $(\CC,\xi)$ is a sequential scale in $\gku$ for some
$\xi$. Let $C=\set\hku{^\gamma}:\gamma^0\leq\gamma\leq\gamma^1.$ be an
fpse-chain of finite type above $\gku$ along $\set
D_\gamma:\gamma^0\leq\gamma\leq\gamma^1.$ relative to $(\CC,\xi)$.

Let $m\in\omega$ and
$\gamma^{(m)}=\set\gamma^0\leq\gamma\leq\gamma^1:
D_\gamma\subseteq\cl{\lasm}^{\kw(\K^{<\gamma})}.$.
Let $\mu<\omega_1$ be such that $h:\mu\to\gamma^{(m)}$ is the unique
1-to-1 monotone map. Let
$C^{(m)}=\set\hku{^\lambda_{(m)}}:\lambda<\mu.$ be the
$m$-reduction of $C$ and let $D\subseteq\lasm$,
$P_0\subseteq\cl{\lasm}^{\kw(\K^\gamma)}$ .

Then $\hd{\K^\gamma}{P_0}{D}=\hd{\K^\lambda_{(m)}}{P_0}{D}$ where
$\lambda=\sup\set\lambda':h(\lambda')\leq\gamma.$. If $\lambda$ is a
limit ordinal and $\gamma\not=h(\lambda)$, $\lambda$ may be replaced
with $<\lambda$.
\end{lemma}
\begin{proof}
Let $P_1\subseteq[D]_\delta^{\kw(\K^\gamma)}$ and $K\in\K^\gamma$ be
such that $P_0\subseteq P_1+K$. By Lemma~\ref{m.reduce.so}
$P_1\subseteq[D]_\delta^{\kw(\K^\lambda_{(m)})}$ and by
Lemma~\ref{ct.compacts} $K=K'+\sum_{\nu\in F}L_\nu$ where
$F\subseteq\gamma+1$ is finite, $K'\in\K$, and
$L_\nu=\cl{D_\nu}^{\tau(\U^\nu)}\subseteq\cl{\la{S^{i(\nu)}}}^{\kw(\K^\nu)}$. Let
$g=p_1-p_0\in K$ for some $p_1\in P_1$ and $p_0\in P_0$. Then
$g\in\cl{\lasm}^{\kw(\K^\gamma)}$ by Lemma~\ref{set.additivity} and
$g=g'+\sum_{\nu\in F}g_\nu$ where $g_\nu\in
L_\nu\subseteq\cl{\la{S^{i(\nu)}}}^{\kw(\K^\gamma)}$ and $g'\in K'$.

As in the proof of Lemma~\ref{m.reduce.so} $g=g''+g_m$ where $g''\in\pi_{I,m}(K')$
and $g_m\in\sum_{i(\nu)=m}L_\nu$ where $I=\set i(\nu):\nu\in
F.\cup\{m\}$. Thus
$p_1-p_0\in\pi_{I,m}(K')+\sum_{i(\nu)=m}L_\nu\subseteq
K''\in\K^\lambda_{(m)}$ so $P_0\subseteq P_1+K''$.
\end{proof}

\begin{lemma}\label{w1.m.hd}
Let $\gku$ be a convenient triple and $\CC=\set S^i:i\in\omega.$ be
such that $(\CC,\xi)$ is a sequential scale in $\gku$ for some
$\xi$. Let $m\in\omega$ and
$C=\set\hku{^\gamma}:\gamma^0\leq\gamma\leq\gamma^1.$ be an fpse-chain
of finite type above $\gku$ away from $\lasm$.
Let $D\subseteq\lasm$ and $P\subseteq\cl{\lasm}^{\kw(\K^\gamma)}$.

Then $P\subseteq G$ and $\hd{\K^\gamma}{P}{D}=\hd{\K}{P}{D}$.
If $P$ is closed and discrete in $\kwk$ then $P$ is closed
and discrete in $\kw(\K^\gamma)$.
\end{lemma}
\begin{proof}
Note that $C^{(m)}=\varnothing$. Applying Lemma~\ref{m.reduce.so} to
each $p\in P$ one shows that 
$P\subseteq G$. 
Let $\hd{\K^\gamma}{P}{D}=\delta$. It follows from
Lemma~\ref{m.reduce.hd} and $C^{(m)}=\varnothing$ that
$\hd{\K}{P}{D}\leq\delta$. To see that $P$ is closed and
discrete note that otherwise for some infinite $P'\subseteq P$ there
exists a compact $K\in\K^{\gamma^1}$ such that $P'\subseteq K$. Just
as in the proof of Lemma~\ref{m.reduce.hd} we may assume that $K\in\K$
contradicting the assumption that $P$ is closed and discrete in $\kwk$.
\end{proof}

\begin{lemma}\label{m.reduce.fpse}
Let $\gku$ be a convenient triple and $\CC=\set S^i:i\in\omega.$ be
such that $(\CC,\xi)$ is a sequential scale in $\gku$ for some
$\xi$. Let $C=\set\hku{^\gamma}:\gamma^0\leq\gamma\leq\gamma^1.$ be an
fpse-chain of finite type above $\gku$ along $\set
D_\gamma:\gamma^0\leq\gamma\leq\gamma^1.$ relative to $(\CC,\xi)$.

Then the $m$-reduction
$C^{(m)}=\set\hku{^\lambda_{(m)}}:\lambda<\mu.$ is an
fpse-chain above $\gku$ along $\set D_{h(\lambda)}:\lambda<\mu.$
relative to $(\CC,\xi)$ close to $\lasm$.
\end{lemma}
\begin{proof}
Let $m\in\omega$ and
$\gamma^{(m)}=\set\gamma^0\leq\gamma\leq\gamma^1:
D_\gamma\subseteq\cl{\lasm}^{\kw(\K^{<\gamma})}.$.
Let $\mu<\omega_1$ be such that $h:\mu\to\gamma^{(m)}$ is the unique
1-to-1 monotone map.

Each $\hku{^\lambda_{(m)}}$ is a convenient triple. To
show~\ref{fine.pse.pse} and~\ref{fine.pse.sep} note that each
$D_{h(\lambda)}\subseteq\cl{\lasm}^{\kw(\K^{<\lambda}_{(m)})}$ by
Lemma~\ref{m.reduce.so} and
$\hd{\K^{<\lambda}_{(m)}}{D_{h(\lambda)}}{P}=\hd{\K^{<h(\lambda)}}{D_{h(\lambda)}}{P}$
for $P\in\{\,\lasm, \lasom\,\}$ by Lemma~\ref{m.reduce.hd} and
$\K^{<\lambda}_{(m)}\subseteq\K^{<h(\lambda)}$.
\end{proof}

The next definition and lemma introduce the concept of stability as
well as a way to pass to a stable subset in some cases.

\begin{definition}\label{stable}
Let $\gku$ be a convenient triple, $\CC=\set S^i:i\in\omega.$ and
$\xi\subseteq\omega_1$ be such that $(\CC,\xi)$ is a sequential scale
in $\gku$. Let $P\subseteq\cl{\lasm}^\kwk$ be closed and discrete in
$\kwk$. Call $P$ {\em $(\CC,\xi)$-stable\/} (or simply {\em stable\/}
if $(\CC,\xi)$ is clear from the context) {\em in $\gku$\/} if for any infinite
$P'\subseteq P$ and any fpse-chain
$\set\hku{^\gamma}:\gamma^0\leq\gamma\leq\gamma^1.$ of finite type
above $\gku$ relative to $(\CC,\xi)$
$\hd{\K^{\gamma^1}}{P'}{\lasom}=\hd{\K}{P}{\lasom}$ provided $P'$ is
closed and discrete in $\kw(\K^{\gamma^1})$.
\end{definition}

\begin{lemma}\label{p.stable}
Let $\gku$ be a convenient triple, $\CC=\set S^i:i\in\omega.$ and
$\xi\subseteq\omega_1$ be such that $(\CC,\xi)$ is a sequential scale
in $\gku$. Let $P\subseteq\cl{\lasm}^\kwk$ be closed and discrete in
$\kwk$. Then there exists an infinite $P'\subseteq P$ and an
fpse-chain $\set\hku{^\gamma}:\gamma^0\leq\gamma<\gamma^1.$ above
$\gku$ close to $\lasm$ such that $P'$ is $(\CC,\xi)$-stable in
$\hku{^{\gamma^1}}$.
\end{lemma}
\begin{proof}
Let $\delta<\omega_1$ be the smallest ordinal such that there exists
an infinite $P'\subseteq P$ and an fpse-chain
$\set\hku{^\gamma}:\gamma^0\leq\gamma<\gamma^1.$ above $\gku$ close
to $\lasm$ such that $P'$ is closed and discrete in
$\kw(\K^{\gamma^1})$ and $\hd{\K^{\gamma^1}}{P'}{\lasom}=\delta$.

If $P''\subseteq P'$ is not stable in $\hku{^{\gamma^1}}$ for some
infinite $P''$, there exists an
fpse-chain $C=\set\hku{^\gamma}:\gamma^1\leq\gamma\leq\gamma^2.$ of
finite type above $\hku{^{\gamma^1}}$ relative to $(\CC,\xi)$ such
that
$\hd{\K^{\gamma^2}}{P'''}{\lasom}<\hd{\K^{\gamma^1}}{P'}{\lasom}=\delta$
for some infinite $P'''\subseteq P''$ closed and discrete in
$\kw(\K^{\gamma^2})$.

Using the $m$-reduction $C^{(m)}$ of $C$ instead of $C$ and
Lemmas~\ref{m.reduce.hd} and~\ref{m.reduce.fpse} leads to a
contradiction with the minimality of $\delta$.
\end{proof}

Call the $\delta$ in the proof above the {\em $(\CC,\xi)$-stable
$\K$-height of $P$ above $\lasom$}.

%\ifdraft\vfill\eject\fi
\section{Graded boolean group extensions}
This section introduces an algebraic mechanism for keeping track and
altering of the sequential order of points.

\begin{definition}
Let $\gku$ be a convenient triple, $\CC=\set S^i:i\in\omega.$ and
$\xi=\set\sigma^i:i\in\omega.$ be such that $(\CC,\xi)$ is a sequential
scale in $\gku$. Let $P\subseteq[\lasom]_\Sigma^\kwk$ for some $m\in\omega$
and $\Sigma<\sigma^m-1$.

Call $W=\set(H^\alpha, \K^\alpha, \U^\alpha, S_\alpha,
s_\alpha):\alpha^0\leq\alpha\leq\Omega.$ a {\em well-graded group stack
above $(P,G,\K,\U,\Sigma)$\/} or {\em wggs\/} for short if it
satisfies the following properties:
\begin{countup}[bcp]
\item\label{wggs.order}
$\set\hku{^\alpha}:\alpha^0\leq\alpha\leq\Omega.$ is an fpse-chain of
finite type above $\gku$ along
$\set S_\alpha:\alpha\leq\Omega.$ relative to $\CC$ close to $\lasm$;

\item\label{wggs.marker}
$S_\alpha\to s_\alpha\in H^\alpha\setminus H^{<\alpha}$ in $\kw(\K^\alpha)$;

\item\label{wggs.ext.limit}
if $\alpha$ is a limit ordinal there exist increasing $\alpha_n<\alpha$ such
that $\alpha_n\to\alpha$,
$\so(s_{\alpha_n}, \lasom, \kw(\K^{\alpha_n}))$ is increasing, and
$S_\alpha=\set s_{\alpha_n}:n\in\omega.$;

\item\label{wggs.ext.successor}
if $\alpha$ is a successor ordinal there exist
$p_n\in P$ such that $S_\alpha=\set p_n:n\in\omega.$;

\item\label{wggs.so.other}
if $g\in H^\alpha\setminus H^{<\alpha}$, $\alpha\leq\Omega$ then
$\so(g, \lasom, \kw(\K^\alpha))\geq\so(s_\alpha, \lasom, \kw(\K^\alpha))$;

\item\label{wggs.top}
$\top(W)=\so(s_{\Omega}, \lasom, \kw(\K^{\Omega}))>\so(s_\alpha, \lasom, \kw(\K^\alpha))$
for any $\alpha<\Omega$;
\end{countup}
\end{definition}

\begin{lemma}\label{wggs.c}
Let $\gku$ be a convenient triple. Let $\CC=\set
S^i:i\in\omega.$ and $\xi=\set\sigma^i:i\in\omega.$ be such that
$(\CC,\xi)$ is a sequential scale in $\gku$.
Let $\set\hku{^\gamma}:\gamma<\alpha.$ be a an fpse-chain above $\gku$
relative to $(\CC,\xi)$ where $\alpha$ is a limit ordinal.

Let $m\in\omega$ and
$S=\set s^n:n\in\omega.\subseteq\cl{\lasom}^{\kw(\K^{<\alpha})}$ be
such that $s_n\in H^{\gamma_n}\setminus H^{<\gamma_n}$ where
$\gamma_n\to\alpha$ is strictly increasing and cofinal in
$\alpha$. Suppose
$\so(g, \lasom, \kw(\K^{\gamma_n}))\geq
\sigma_n<\sigma^m$ 
for any $g\in H^{\gamma_n}\setminus H^{<\gamma_n}$. Then $S$ is
independent,
$\hd{\K^{<\alpha}}{S}{\lasom}\geq\sigma=\sup\sigma_n$ and $\las$ is
closed and discrete in $\kw(\K^{<\alpha})$. In particular, if
$\so(s^n, \lasom, \kw(\K^{\gamma_n}))=\sigma_n$ then
$\hd{\K^{<\alpha}}{S}{\lasom}=\sigma$.
\end{lemma}
\begin{proof}
Let $K\in\K^{<\alpha}$ and $\beta<\sigma$. Then $K\in\K^\gamma$ for some
$\gamma<\alpha$. Let $\gamma_i>\gamma$ be such that
$\so(s_i, \lasom, \kw(\K^{\gamma_i}))=\sigma_i>\beta$. If
$a+g=s_i$ for some $a\in K$ and $g\in H^{<\alpha}$ then $g\in
H^{\gamma_i}\setminus H^{<\gamma_i}$. Then either
$\so(g, \lasom, \kw(\K^{<\alpha}))\geq\sigma^m>\sigma_i>\beta$ or
$\so(g, \lasom, \kw(\K^{<\alpha}))=\so(g, \lasom, \kw(\K^{\gamma_i}))\geq 
\so(s_i, \lasom, \kw(\K^{\gamma_i}))=\sigma_i>\beta$ by the assumption
and Lemma~\ref{fpsestack}.

Thus for any such $K$ the set
$S\setminus (K+[\lasom]_\beta^{\kw(\K^{<\alpha})})\not=\varnothing$ for
every $\beta<\sigma$ so $\hd{\K^{<\alpha}}{S}{\lasom}\geq\sigma$.

If $s\in\las$ then $s=\sum_{n\in I}s_n$ for some finite
$I\subseteq\omega$ and $s\in H^{\gamma'}\setminus H^{<\gamma'}$ where
$\gamma'=\max_{n\in I}\gamma_n$ since $\gamma_n$ are strictly
increasing. Thus if $s\not\in\laf$ where $F=\set
s_n:\gamma_n\leq\gamma.$ then $s\not\in K$. Therefore $\las\cap K$ is
finite and $\las$ is closed and discrete in $\kw(\K^{<\alpha})$.
\end{proof}

\begin{lemma}\label{wggs.w}
Let $\gku$ be a convenient triple. Let $\CC=\set
S^i:i\in\omega.$ and $\xi=\set\sigma^i:i\in\omega.$ be such that
$(\CC,\xi)$ is a sequential scale in $\gku$. Let $m\in\omega$ and
$P=\set p^n:n\in\omega.\subseteq[\lasom]_\Sigma^\kwk$ be an
independent subset of $G$ such that $\la{P}$ is closed and discrete in
$\kwk$, $\hd{\K}{P}{\lasom}=\Sigma<\sigma^m-1$, and $P$ is
$(\CC,\xi)$-stable in $\gku$. Let $D=\set
d^n:n\in\omega.\subseteq G$ and $d\in[D]_\mu^\kwk$ where $\mu\geq1$ is
a successor ordinal, $\Sigma+\mu\leq\sigma\leq\sigma^m-1$ for some
successor ordinal $\sigma$.

Then there exists a wggs $W=\set(H^\alpha, \K^\alpha, \U^\alpha,
S_\alpha, s_\alpha):\alpha\leq\Omega.$ above $(P, G, \K, \U, \Sigma)$
such that $\top(W)=\sigma$ and the following property holds.
\begin{countup}[bcp]
\item\label{wggs.w.align}
$d+s_\Omega\in[\set d^n+p^n:n\in\omega.]_\eta^{\kw(\K^\Omega)}$ where
$$
\eta=\max\{\mu,\so(s_\Omega, P, \kw(\K^\Omega))\}\leq\sigma-\Sigma;
$$

\end{countup}
In addition, if $P'\subseteq G$ is such that $P'\cap P=\varnothing$,
$P\cup P'$ is independent, and $\langle P'\cup P\rangle$ is closed and
discrete in $\kwk$ then one may assume the following property holds.
\begin{countup}[bcp]
\item\label{wggs.i.discrete}
$\la{P'}$ is closed and discrete in $\kw(\K^\Omega)$;

\end{countup}

\end{lemma}
\begin{proof}
Let $\sigma=\Sigma+1<\sigma^m$. Then $\mu=1$ so by passing to subsets
and reindexing we may assume that $d^n\to d$ in $\kwk$. Use
Lemma~\ref{ct.split} to find an $S_0\subseteq P$, and $\U^0\supseteq\U$
such that $\cl{\la{S'}}^{\tau(\U^0)}\cap G=\la{S'}$ for every infinite
$S'\subseteq S_0$ and
$\cl{\le{S_0}}^{\tau(\U^0)}\cap\cl{\lo{S_0}}^{\tau(\U^0)}=\varnothing$. By passing
to subsets and applying Lemma~\ref{free.sequence} we may assume that
$S_0\to s_0$ in $\tau(\U^0)$ and $S_0$ satisfies \MM{\lasom, \K}
and \MM{\lasi, \K} for every $i\in\omega$. Let $\hku{^0}$ be the fpse
of $\gku$ over $S_0$. Put $\hku{^{<0}}=\gku$ and
$W=\set(H^i, \K^i, \U^i, S_i,
s_i):i<1.$. Properties~\ref{wggs.order}--\ref{wggs.ext.successor} are
immediate, \ref{wggs.top} holds vacuously,
and~\ref{wggs.w.align} holds by the construction.

Let $g\in H_0\setminus G$ and let
$\so(g, \lasom, \kw(\K^0))=\sigma'$. If $\sigma'\leq\Sigma$
then $\so(g, \lasom, \kwk)=\sigma'$ by Lemma~\ref{D.flat}, in
particular $g\in G$. Thus $\sigma'>\Sigma$ so
$\sigma'\geq\Sigma+1\geq\so(s_0, \lasom, \kw(\K^0))$
and~\ref{wggs.so.other} holds.

Suppose $K\in\K$ is such that $(K+L)\cap\la{P'}$ is infinite where
$L=\cl{\la{S_0}}^{\tau(\U^0)}$. Let $a_n+d_n=p_n\in\la{P'}$ be
distinct points in $\la{P'}$ such that $a_n\in K$ and $d_n\in L$. Then
$d_n\in G$ so by~\ref{pse.resolve} $d_n\in\la{S_0}\subseteq\la{P}$. If
the set $\set d_n:n\in\omega.$ is infinite then $K\cap\la{P'\cup P}$
is infinite since $P'$ and $P$ are disjoint and $P'\cup P$ is
independent, contradicting $\la{P'\cup P}$ is closed and discrete in
$\kwk$. Thus $\set d_n:n\in\omega.$ is finite so $K'\cap\la{P'}$ is
infinite for some $K'\in\K$ contradicting $\la{P'\cup P}$ is closed
and discrete in $\kwk$, so~\ref{wggs.i.discrete} holds.

Suppose $W$ that satisfy~\ref{wggs.w.align} and~\ref{wggs.i.discrete}
can be constructed for every $\sigma'<\sigma$ where
$\Sigma+1<\sigma\leq\sigma^m-1$.

If $\mu>1$ let $\mu_n\to\mu-1$ be increasing such that $d_n'\to d$ in
$\kw(\K)$ where $\so(d_n', D, \kwk)=\mu_n$. Using regularity, construct disjoint
$I_n\subseteq\omega$ such that $\so(d_n', D_n, \kwk)=\mu_n$ where
$D_n=\set d^i:i\in I_n.$.

If $\mu=1$ assume, by
possibly thinning out and reindexing that $d^n\to d$, put
$\mu_n=1$ and let $I_n\subseteq\omega$ be arbitrary infinite disjoint
subsets.

Let $\sigma_i\to\sigma-1>\Sigma$ be successor ordinals such that
$\sigma_i\geq\Sigma+\mu_i$. Let $P_n=\set p^i:i\in I_n.$. Note that
every subset of $P$ (including $P_n$, and their arbitrary unions) is
$(\CC,\xi)$-stable in $\gku$ and
$\hd{\K}{P_n}{\lasom}=\Sigma$.

Build, by induction, a sequence of
$W_i=\set(H^\alpha, \K^\alpha, \U^\alpha, S_\alpha,
s_\alpha):\gamma_{i-1}<\alpha\leq\gamma_i.$ such that $W_0$ is a wggs above
$(P_0,G,\K,\U,\Sigma)$, $W_{i+1}$ is a wggs above
$(P_{i+1},H^{\gamma_i},\K^{\gamma_i},\U^{\gamma_i},\Sigma)$ for
$i\in\omega$, and the following properties hold.
\begin{countup}[bcp]
\item\label{wggs.i.group}
the subgroup $\la{P'\cup\cup_{j>i}P_j}$ is closed and discrete in
 $\kw(\K^{\gamma_i})$;

\item\label{wggs.i.so}
for each $i\in\omega$
$\so(s_{\gamma_i},\lasom,\kw(\K^{\gamma_i}))=\sigma_i$;

\item\label{wggs.i.sum}
$d_n'+s_{\gamma_n}\in[\set d^i+p^i:i\in
I_n.]_{\eta_n}^{\kw(\K^{\gamma_i})}$ where
$$
\eta_n=\max\{\mu_n, \so(s_{\gamma_n}, P_n, \kw(\K^{\gamma_n}))\leq\sigma_i-\Sigma;
$$

\end{countup}
At the $n$-th step the set $\cup_{j>n}P_j$ is closed, discrete, and
$(\CC,\xi)$-stable in $\hku{^{\gamma_n}}$,
$\hd{\K^{\gamma_n}}{\cup_{j>n}P_j}{\lasom}=\Sigma$, and $\sigma_i<\sigma$ so $W_{n+1}$ exists
by the hypothesis.

Let $\Omega=\lim\gamma_i$ and consider the conventient triple $\hku{^{<\Omega}}$.
Put $S_\Omega'=\set s_{\gamma_i}:i\in\omega.$ and let $g\in
H^{<\Omega}$ be such that $g\in H^\alpha\setminus H^{<\alpha}$ and
$\so(g, \lasom, \kw(\K^{<\Omega}))=\sigma'$ for some $\sigma'<\sigma^m$
and $\alpha<\Omega$. Then $\so(g, \lasom, \kw(\K^\alpha))=\sigma'$
by~\ref{wggs.order}, Lemma~\ref{fpsestack}, and the construction of
$W_i$, in particular,
$\so(s_{\gamma_i}, \lasom, \kw(\K^{<\Omega}))=\sigma_i$.

The sequence of ordinals $\gamma_i$ is cofinal in
$\Omega$ and $\so(g, \lasom, \kw(\K^{\gamma_i}))\geq\sigma_i$ for
every $g\in H^{\gamma_i}\setminus H^{<\gamma_i}$
by~\ref{wggs.so.other} and the hypothesis so
$\hd{\K^{<\Omega}}{S_\Omega}{\lasom}=\sigma-1$ for any infinite subset
$S_\Omega\subseteq S_\Omega'$ and $\la{S_\Omega'}$ is closed and
discrete in $\kw(\K^{<\Omega})$ by Lemma~\ref{wggs.c}. 

Construct an fpse $\hku{^\Omega}$ of $\hku{^{<\Omega}}$ over some
infinite $S_\Omega\subseteq S_\Omega'$ such that $S_\Omega\to
s_\Omega$ in $\kw(\K^\Omega)$ using Lemma~\ref{fpse.successor} and let
$L=\cl{\la{S_\Omega}}^{\tau(\U^\Omega)}$.

Suppose $\so(g', \lasom, \kw(\K^\Omega))<\sigma$ for some $g'\in
H^\Omega$. Then $\so(g', \lasom, \kw(\K^\Omega))\leql\sigma-1$ so by
Lemma~\ref{D.flat}
$\so(g', \lasom, \kw(\K^{<\Omega}))\leql\sigma-1<\omega_1$ and $g'\in
H^{<\Omega}$. Thus $\so(g, \lasom, \kw(\K^\Omega))\geq\sigma$ for any
$g\in H^\Omega\setminus H^{<\Omega}$, in particular,
$\so(s_\Omega, \lasom, \kw(\K^\Omega))=\sigma$.

Now~\ref{wggs.order}--\ref{wggs.top} hold by
construction, \ref{wggs.w.align} holds by the choice of $d_n'$.

Let $K\in\K^{<\Omega}$ be such that $K+L\cap\la{P'}$ is
infinite. Suppose $p_n=a_n+s_n$ are distinct points in $\la{P'}$ such
that $a_n\in K$ and $s_n\in L$. Then $s_n\in H^{<\Omega}$ so
$s_n\in\la{S_\Omega}$ by~\ref{pse.resolve}. Let $K\in\K^\beta$ where
$\beta<\Omega$. If the set $\set s_n:n\in\omega.$ is infinite there is
an $n\in\omega$ such that $s_n\in H^\alpha\setminus H^{<\alpha}$ for
some $\alpha\geq\beta$ so $p_n=a_n+s_n\not\in G$ contradicting $p_n\in
P'$. Thus $\set s_n:n\in\omega.$ is finite and $p_n\in
K'\in\K^{\gamma_i}$ for some $\gamma_i<\Omega$. Then $\la{P'}$ is not
closed and discrete in $\kw(\K^{\gamma_i})$ contradicting the
properties of $W_i$. Hence~\ref{wggs.i.discrete} holds.
\end{proof}

\begin{lemma}\label{cl.projection}
Let $\CC=\set S^i:i\in\omega.$ be discrete separated in some $\kw$-pair
$(G,\K)$. Let $g\in\cl{P}^\kwk$ for some $g\in G$ and $P=\set
p_j:j\in\omega.\subseteq G$. Let $D=\set d_j:j\in\omega.$ where
$d_j=\sum_{i\in I}d_j^i$ for some finite $I\subseteq\omega$, and
$d_j^i\in\cl{\lasi}^\kwk$. Suppose $p_j-d_j\in K$ for some $K\in\K$
and every $j\in\omega$.

Then there are $d^i\in\cl{\lasi}^\kwk$, $i\in I$ and $a\in K$ such
that
\begin{countup}[bcp]
\item\label{cl.projection.sum}
$g=a+\sum_{i\in I}d^i$;

\item\label{cl.projection.align}
for any $U$, $U_K$, $U(i)$, $i\in I$ open in $\kwk$ such that $g\in
U$, $a\in U_K$, and $d^i\in U(i)$ there is an $n\in\omega$ such that
$p_n\in U$, $p_n-d_n\in U_K$, and $d_n^i\in U(i)$;

\end{countup}
\end{lemma}
\begin{proof}
Suppose the Lemma holds for all $g'\in G$ such that $\so(g',
P, \kwk)=\sigma'<\sigma$ and let $\so(g, P, \kwk)=\sigma$. If
$\sigma=0$ the proof is immediate so assume $\sigma>0$. Then there are
$g_n\in[P]_{\sigma_n}$ where $\sigma_n<\sigma$ and $g_n\to g$. By the
hypothesis there are $a_n\in K$ and $d^{i,n}\in\cl{\lasi}^\kwk$ for
$i\in I$ such that $g_n=a_n+\sum_{i\in I}d^{i,n}$ and for any
$m\in\omega$ and any open in $\kwk$ subsets $U\ni g_m$, $U_K\ni a_m$,
$U(i)\ni d^{i,m}$ where $i\in I$ there is an $n\in\omega$ such that
$p_n\in U$, $d_n^i\in U(i)$, $p_n-d_n\in U_K$.

Since $\set\sum_{i\in I}d^{i,n}:n\in\omega.\subseteq K'$ for
some $K'\in\K$ and $\CC$ is discrete separated by $\K$ in $G$,
there is an $i\in I$ such that the set $\set d^{i,n}:n\in\omega.$ is not
closed and discrete. By thinning out and reindexing we then may assume
that $d^{i,n}\to d^i$ for some $d^i\in\cl{\lasi}^\kwk$. Replacing $K'$
with some $K''\supseteq K'+\set d^{i,n}:n\in\omega.$ and repeating this
argument we may find $d^i$ such that $d^{i,n}\to d^i$ for each $i\in
I$. By thinning out and reindexing we may assume that $a_n\to a$ for some
$a\in K$.
Then $g=a+\sum_{i\in I}d^i$.

Let $U$, $U_K$, $U(i)$, $i\in I$ be open in $\kwk$ such that $g\in U$,
$a\in U_K$, and $d^i\in U(i)$ for every $i\in I$. Pick an $m\in\omega$
such that $g_m\in U$, $a_m\in U_K$, and $d^{i,m}\in U(i)$ for every $i\in I$. By the
hypothesis there is an $n\in\omega$ such that $p_n\in U$, $p_n-d_n\in
U_K$, and $d_n^i\in U(i)$, $i\in I$.
\end{proof}

\begin{lemma}\label{g.kw}
Let $\xi=\set\sigma^m:m\in\omega.\subseteq\omega_1\setminus 2$, $m\in\omega$ be a
family of successor ordinals. Then there exists a convenient triple
$(G, \K, \U)$ and a family $\CC=\set S^i:i\in\omega.$ of discrete
separated subsets of $G$ such that $\so(0, \lasom, \kwk)=\sigma^m$ for
every $m\in\omega$ and $(\CC,\xi)$ is a sequential scale in $\gku$.
\end{lemma}
\begin{proof}
For each $m\in\omega$ let $\gks{(m)}$ be a convenient triple such that
$\kw(\K(m))$ is neither discrete nor compact.
Let $\CC_m=\set S(n):n\in\omega.$ be a family of countable disjoint
subsets of $G(m)\setminus\{0\}$ such that $\cup_{n\in\omega}S(n)$ is an
independent set and $\la{\cup_{n\in\omega}S(n)}$ is closed and
discrete in $\kw(\K(m))$. Then $(\CC_m,\xi)$ is a trivial sequential scale in
$\gks{(m)}$.

Let $\sigma^m_n\to\sigma^m-1$.
Let $S(m)=\cup_{n\in\omega}S(m,n)$ where $S(m,n)$ are infinite and
disjoint. Each $S(m,n)$ is trivially $(\CC_m,\xi)$-stable in $\gks{(m)}$
(its $\K(m)$-height is $0$). Put $\gks{^{-1}}=\gks{(m)}$,
$\gamma_{-1}(m)=-1$, and use
Lemma~\ref{wggs.w} to construct wggs
$$
W_n(m)=\set(G^\gamma, \K^\gamma, \U^\gamma, S_\gamma,
s_\gamma):\gamma_{n-1}(m)\leq\gamma<\gamma_n(m).
$$
above $(S(m,n),
G^{\gamma_{n-1}(m)}, \K^{\gamma_{n-1}(m)}, \U^{\gamma_{n-1}(m)}, 0)$
relative to $(\CC_m,\xi)$ such that
$\top(W_n(m))=\sigma^m_n$ for every $n\in\omega$. Let
$\alpha(m)=\lim\gamma_n(m)$.

Just as in the proof of Lemma~\ref{wggs.w}, the set $S=\set
s_{\gamma(n)}:n\in\omega.$ is closed and discrete and
$\hd{\K^{<\alpha(m)}}{S}{\lo{S(m)}}=\sigma^m-1$. Use
Lemma~\ref{fpse.successor} to find an fpse $\gks{^\alpha(m)}$ of
$\gks{^{<\alpha(m)}}$ such that $S\to s\not\in G^{<\alpha(m)}$ in
$\kw(\K^{\alpha(m)})$. Then $\so(s, \lo{S(m)}, \kw(\K^{\alpha(m)}))=\sigma^m-1$,
the set $s+S(m)$ is independent,
$\la{s+S(m)}\setminus\le{(s+S(m))}=s+\lo{S(m)}$ so $\S(\K^{\alpha(m)},
s+S(m), \sigma^m)$ holds. Put $S^m=s+S(m)$, $\gks{_m}=\gks{^{\alpha(m)}}$.

Let $G$ be the direct sum of $G_m$, $\U$ be the appropriate basis of
the topology inherited from the product of $(G_m,\tau(\U_m))$, and
$\K$ be the closure of $\cup_{m\in\omega}\K_m$ under finite sums.
Put $\CC=\set S_m:m\in\omega.$. Then the desired properties hold by
the construction.

Note that in the proof above, the full sequential scale $\CC_m$ was
only required to formally satisfy the conditions of
Lemma~\ref{wggs.w}.
\end{proof}

Let $\sigma^m$ be an increasing sequence of successor ordinals in
$\omega_1$. Let $\sigma^i_k\to\sigma^i-1$ be an increasing sequence of
successor
ordinals for every $i\in\omega$ such that
$\sup\sigma^i_k=\sigma^i-1$.

Since all groups $G_\alpha$ in the construction have cardinality
$2^\omega$ we will assume that every $G_\alpha$ is a subgroup
(algebraically) of $2^\omega$. Let $\set
C_\alpha:\alpha<\omega_1.\subseteq[2^\omega]^\omega$ and $\set
P_\alpha:\alpha<\omega_1\hbox{ is a
limit ordinal}.\subseteq[2^\omega]^\omega$.

Let $\gks{_0}=\gku$ where the convenient triple $\gku$ and the
sequential scale $(\CC,\xi)$ have been constructed in Lemma~\ref{g.kw}.
Below we use the notation
$\ccl{C_\alpha}$ for the closure of $C_{\mcofl(\alpha)}$ in $\kw(\K_{<\mcofl(\alpha)})$.

\begin{lemma}\label{g.ind}
There exists an fpse-chain $\set\gks{_\alpha}:\alpha<\omega_1.$ relative to
$(\CC,\xi)$ and an increasing continuous $\mcofl:\omega_1\to\omega_1$ such that
\begin{countup}[bcp]
\item\label{g.cc}
if $\alpha$ is a limit ordinal and $P_\alpha\subseteq
G_{\mcofl(\alpha+1)}$ is an infinite subset that is closed and
discrete in $\kw(\K_{\mcofl(\alpha+1)})$ then
there exists an infinite $S_\alpha\subseteq P_\alpha$ such that
$S_\alpha\to s_\alpha$ in $\kw(\K_{\mcofl(\alpha+2)})$;

\item\label{g.so}
if $\alpha$ is a limit ordinal and
$0\in\overline{P_\alpha}^{\kw(\K_{\mcofl(\alpha+1)})}$ then $\so(0,
P_\alpha, \kw(\K_{\mcofl(\alpha+2)}))\leq\sigma^m$ for some $m\in\omega$;

\item\label{g.lus}
if $\alpha$ is a limit ordinal then one of the following conditions holds
\begin{countup}[label={\rm\alph*.},ref={\rm(\arabic{countupi}.\alph*)}]
\item\label{g.lus.c}
if $\hd{\K_{<\mcofl(\alpha)}}{\ccl{C_\alpha}\cap U}{\CC}=\omega$ for every
$U\in\U_{<\mcofl(\alpha)}$ then $S\to0$ in $\kw(\K_{\mcofl(\alpha)})$ for some
$S\subseteq \ccl{C_\alpha}$;

\item\label{g.lus.wggs}
suppose $\ccl{C_\alpha}\subseteq\sum_{i\in
I}\cl{\lasi}^{\kw(\K_{<\mcofl(\alpha)})}$ for some finite
$I\subset\omega$ (so \ref{g.lus.c} does not hold), and $\ccl{C_\alpha}\cap\sum_{i\in
I'}\cl{\lasi}^{\kw(\K_{<\mcofl(\alpha)})}=\varnothing$ for any $I'\subset
I$, $I'\not=I$; if there exists an fpse-chain
$C=\set\gks{_\gamma}:\mcofl(\alpha)\leq\gamma\leq\gamma^0.$ such that
$\cl{\ccl{C_\alpha}}^{\kw(\K_{\gamma^0})}\cap\sum_{i\in
I'}\cl{\lasi}^{\kw(\K_{\gamma^0})}\not=0$ for some $I'\subset I$,
$I'\not= I$; then $\cl{\ccl{C_\alpha}}^{\kw(\K_{\mcofl(\alpha+1)})}\cap\sum_{i\in
I''}\cl{\lasi}^{\kw(\K_{\mcofl(\alpha+1)})}\not=0$ for some $I''\subset
I$, $I''\not= I$;

\item\label{g.lus.so.raise}
suppose neither \ref{g.lus.c} nor~\ref{g.lus.wggs} holds,
$\ccl{C_\alpha}\subseteq\sum_{i\in
I}\cl{\lasi}^{\kw(\K_{<\mcofl(\alpha)})}$ for some finite $I\subset\omega$, and
$\ccl{C_\alpha}\cap\sum_{i\in
I'}\cl{\lasi}^{\kw(\K_{<\mcofl(\alpha)})}=\varnothing$ for any
$I'\subset I$, $I'\not=I$; suppose there exists an fpse-chain
$$
C=\set\gks{_\gamma}:\mcofl(\alpha)\leq\gamma\leq\gamma^0.
$$
and
for some $\gamma_n$, $n\in\omega$, strictly increasing and cofinal in
$\gamma^0$ the following properties hold:

there are $d(n)\in\cl{\ccl{C_\alpha}}^{\kw(\K_{\gamma^0})}$ such that
$d(n)=\sum_{i\in I}d(i, n)$ for every $n\in\omega$ where
$d(i,n)\in\cl{\lasi}^{\kw(\K_{\gamma^0})}\cap
\bigcap\set \cl{U}^{\tau(\U_{\gamma^0})}:U\in\U_{<\mcofl(\alpha)}.$,
and

$\so(g, \la{S^{i(n)}}, \kw(\K_{\gamma^0}))\geq
\sigma^{i(n)}_n$
for any $g\in G_{\gamma_n}\setminus G_{<\gamma_n}$ where $d({i(n)},n)\in 
G_{\gamma_n}\setminus G_{<\gamma_n}$ for some choice of $i(n)\in I$
and $d(i,n)\not\in G_{<\mcofl(\alpha)}$ for $i\in I$;
then the condition above holds if
$\gamma^0$ is replaced by $\mcofl(\alpha+1)$.

\end{countup}

\item\label{g.resolve}
if $\alpha$ is a limit ordinal then for every finite
$I\subset\omega$
$$
\cap\set U\in\U_{\mcofl(\alpha+3)}:H(I)\subseteq U.=H(I)
$$
where $H(I)=\sum_{i\in I}
\cl{\lasi}^{\kw(\K_{\mcofl(\alpha+3)})}$
and for any $U\in\U_{\mcofl(\alpha+3)}$ there is a $U'\subseteq U$ such
that $U'\in\U_{\mcofl(\alpha+3)}$ and
$$
U'\cap(\sum_{i\in
I}\cl{\lasi}^{\kw(\K_{\mcofl(\alpha+3)})})=\sum_{i\in
I}(U\cap\cl{\lasi}^{\kw(\K_{\mcofl(\alpha+3)})});
$$

\end{countup}
\end{lemma}
\begin{proof}
Put $\mcofl(0)=0$. If $\alpha$ is a
successor such that $\alpha=\beta+n$ where $n\in\omega$ and $\beta$ is
a limit and $n>3$ or $\beta=0$ and $n>0$ let
$\mcofl(\alpha)=\mcofl(\alpha-1)+1$ and let $\gks{_{\mcofl(\alpha)}}$
be an arbitrary fpse of $\gks{_{\mcofl(\alpha-1)}}$.

Let $\alpha$ be a limit ordinal and $G_{\alpha'}$ etc., $\mcofl(\alpha')$
have been constructed for all $\alpha'<\alpha$. Then $\mcofl(\alpha)$ is
uniquely defined by continuity.

Let $\K_{<\mcofl(\alpha)}=\set K_i:i\in\omega.$ and $\U_{<\mcofl(\alpha)}=\set
U_i:i\in\omega.$.

Suppose $\hd{\K_{<\mcofl(\alpha)}}{\ccl{C_\alpha}\cap U}{\CC}=\omega$ for every
$U\in\U_{\mcofl(\alpha)}$. Find an infinite $S\subseteq \ccl{C_\alpha}$ such that
$S\to0$ in $\tau(\U_{<\mcofl(\alpha)})$ and $S$
satisfies \MM{\CC, \K_{<\mcofl(\alpha)}} using
Lemma~\ref{free.sequence}. Use Lemma~\ref{z.group} to find an
fpse $\gks{_{\mcofl(\alpha)}}$ of $\gks{_{<\mcofl(\alpha)}}$ over $S$ such that
$S\to0$ in $\K_{\mcofl(\alpha)}$. Let $\gks{_{\mcofl(\alpha+1)}}$ be
an arbitrary fpse of $\gks{_{\mcofl(\alpha)}}$ and put
$\mcofl(\alpha+1)=\mcofl(\alpha)+1$ to satisfy~\ref{g.lus.c}.

If the wggs from the statement of~\ref{g.lus.wggs}
or~\ref{g.lus.so.raise} exists put
$\mcofl(\alpha+1)=\gamma^0$. 

Suppose $P_\alpha$ is closed and discrete in
$G_{\mcofl(\alpha+1)}$.

If
$\hd{\K_{\mcofl(\alpha+1)}}{P_\alpha}{\CC}=\omega$ use
Lemma~\ref{fpse.successor} to find an fpse $\gks{_{\mcofl(\alpha+2)}}$
of $\gks{_{\mcofl(\alpha+1)}}$ over $S_\alpha\subseteq P_\alpha$ such
that $S_\alpha\to s\in G_{\mcofl(\alpha+2)}$ in
$\kw(\K_{\mcofl(\alpha+2)})$ where $\mcofl(\alpha+2)=\mcofl(\alpha+1)+1$.

Otherwise $P_\alpha\subseteq\sum_{i\leq
m}\cl{\lasi}^{\kw(\K_{\mcofl(\alpha+1)})}+K$ for some
$K\in\K_{\mcofl(\alpha+1)}$. Pick an infinite $D=\set d_n=\sum_{i\leq
m}d^i_n:n\in\omega.$ where
$d^i_n\in\cl{\lasi}^{\kw(\K_{\mcofl(\alpha+1)})}$ and $p_n=d_n+a_n$
for some distinct $p_n\in P_\alpha$ and $a_n\in K$. By passing to
subsequences and reindexing we may assume that $a_n\to a$ for some
$a\in K$ and either $d^i_n\to d^i$ or $D^i=\set d^i_n:n\in\omega.$ is
closed and discrete in $\K_{<\mcofl(\alpha)}$. Repeatedly using
Lemma~\ref{fpse.successor} find a finite fpse-chain
$$
W=\set\gks{_\gamma}:\mcofl(\alpha+1)<
\gamma\leq\mcofl(\alpha+1)+k=\mcofl(\alpha+2).
$$  
for some $k\leq m$ such that $\set d^i_n:n\in J.\to d^i$ in
$\kw(\K_{\mcofl(\alpha+2)})$ for some infinite $J\subseteq\omega$.

Let $0\in\cl{P_\alpha}^{\kw(\K_{\mcofl(\alpha+1)})}$ and suppose
$\ha{\K_{\mcofl(\alpha+1)}}{P_\alpha}{\CC}<\omega$. Let $P_\alpha=\set
p_j:j\in\omega.$ and
$p_j-d_j\in K$ for some $K\in\K_{\mcofl(\alpha+1)}$ where $d_j=\sum_{i\in
I}d_j^i$ for some finite $I\subset\omega$, and
$d_j^i\in\cl{\lasi}^{\kw(\K_{\mcofl(\alpha+1)})}$.

Using Lemma~\ref{cl.projection} find
$d^i\in\cl{\lasi}^{\kw(\K_{\mcofl(\alpha+1)})}$ and $a\in K$ such that
$0=a+\sum_{i\in I}d^i$ and for any open in $\kw(\K_{\mcofl(\alpha+1)})$
subsets $U\ni0$, $U_K\ni a$, $U(i)\ni d^i$, $i\in I$ there is an $n\in\omega$ such
that $p_n\in U$, $d_n^i\in U(i)$, and $p_n-d_n\in U_K$, $i\in I$.

Replacing each $d_n^i$ with $d_n^i+d^i$, each $d_n$ with $d_n+a$ and picking
$K\in\K_{\mcofl(\alpha+1)}$ so that $d^i\in K$, $i\in I$, we may assume
that the following properties hold.
\begin{countup}[bcp]
\item\label{p.project}
$d_j=\sum_{i\in I}d_j^i$,
$d_j^i\in\cl{\lasi}^{\kw(\K_{\mcofl(\alpha+1)})}$, and $p_j-d_j\in K$ for
every $j\in\omega$;

\item\label{p.cl}
for any $U\ni0$ open in $\kw(\K_{\mcofl(\alpha+1)})$ there exists an
$n\in\omega$ such that $p_n,d_n^i,p_n-d_n\in U$
for every $i\in I$;

\end{countup}
Indeed, let $U$ be open in $\kw(\K_{\mcofl(\alpha+1)})$. Find a $U'\ni0$,
open in $\kw(\K_{\mcofl(\alpha+1)})$ such that $U'+\cdots+U'\subseteq
U$ where the sum has $|I|$ terms. Put $U_K=U'+a$, $U(i)=U'+d^i$ and
let $n\in\omega$ be such that $p_n\in U'$, $d_n^i\in U(i)$, and
$p_n-d_n\in U_K$. Then $d_n^i+d^i\in U'$ and $p_n-(d_n+a)=p_n-d_n+a\in
U'$.

Let $\gks{(0)}=\gks{_{\mcofl(\alpha+1)}}$. By induction on $k\in\omega$
build convenient triples $\gks{(k)}$ and points
$d(i,k)\in\cl{\lasi}^{\kw(\K(k))}$ where $i\in I$, such that the
following properties
hold for $k>0$:
\begin{countup}[bcp]
\item\label{p.k.fpse}
$\gks{(k)}=\gks{_{\gamma(k)}}$ for some $\gamma(k)<\omega_1$ where
$$
\set\gks{_{\gamma}}:\gamma(k-1)<\gamma\leq\gamma(k).
$$
is an fpse-chain above
$\gks{(k-1)}$ relative to $(\CC,\xi)$;

\item\label{p.k.d}
either $d(i,k)=0$ or $d(i,k)\in G_{\gamma(i,k)}\setminus
G_{<\gamma(i,k)}$ where $\gamma(k-1)<\gamma(i,k)\leq\gamma(k)$;

\item\label{p.k.so.bound}
if $g\in G_{\gamma(i,k)}\setminus G_{<\gamma(i,k)}$ then
$\so(g, \lasoi, \K_{\gamma(i,k)})\geq\sigma^i_k$;

\item\label{p.k.gshift}
$d(i,k)\in\cl{U}^{\tau(\U(k))}$ and $p(k)-(\sum_{i\in I}d(i,k))\in
K\cap\cl{U}^{\tau(\U(k))}$ for every $U\in\U(k-1)$ where
$p(k)\in[P_\alpha]_{\eta(k)}^{\kw(\K(k))}$, and
$\eta(k)\leq\max_{i\in I}\sigma^i-1$;

\end{countup}
Suppose $\gks{(k')}$ etc.\ that satisfy~\ref{p.k.fpse}--\ref{p.k.gshift}
have been built for all $k'<k$. Let $\K(k-1)=\set
K_n:n\in\omega.$ and $\U(k-1)=\set U_n:n\in\omega.$. Pick
$n(j)\in\omega$ by induction using~\ref{p.project} and~\ref{p.cl} so
that
$$
p_{n(j)}, d^i_{n(j)}\in\cap_{n<j}U_n\hbox{ and }p_{n(j)}-\sum_{i\in
I}d^i_{n(j)}\in K\cap\cap_{n<j}U_n
$$
In the rest of the construction below an infinite set $J\subseteq\set
n(j):j\in\omega.$ will be selected so that $D(J,i)=\set d_j^i:j\in J.$
are independent subsets that satisfy some additional properties. Note that
$D(J',i)\to0$ in $\tau(\U(k-1))$ for every $i\in I$ and any infinite
$J'\subseteq J$.

Let $i\in I$. If $D(J,i)$ is not closed and discrete in $\kw(\K(k-1))$
then by~\ref{p.cl} there exists an infinite $J_i\subseteq J$ such that
$D(J_i,i)\to0$ in $\kw(\K(k-1))$. Otherwise one can find an infinite
$J_i\subseteq J$ such that for any infinite $J'\subseteq J_i$
$\hd{\K(k-1)}{D(J_i,i)}{\lasoi}=\hd{\K(k-1)}{D(J',i)}{\lasoi}$. 

Repeating this argument for every $i\in I$ in the natural order one
can build $J_i$ such that $J_{i'}\subseteq J_i$ if $i'\leq i$ and one of
the two alternatives above holds. After replacing $J$ with $J=\cap_{i\in
I}J_i$ we may assume that either $D(J,i)\to0$ in $\kw(\K(k-1))$ or $D(J,i)$ is closed and
discrete in $\kw(\K(k-1))$ and
$\hd{\K(k-1)}{D(J,i)}{\lasoi}=\hd{\K(k-1)}{D(J',i)}{\lasoi}$ for any
infinite $J'\subseteq J$.

Let $I=I_u\cup I_z\cup I_h$ where $D(J,i)\to0$ in $\kw(\K(k-1))$ for
every $i\in I_u$, $D(J,i)$ is closed
and discrete in $\kw(\K(k-1))$ for $i\in I_z\cup I_h$,
$\hd{\K(k-1)}{D(J',i)}{\lasoi}\geq\sigma^i-1$ for
each infinite $J'\subseteq J$ and $i\in I_z$, and
$\hd{\K(k-1)}{D(J,i)}{\lasoi}<\sigma^i-1$ for every $i\in I_h$.

Passing to a subset of $J$ if necessary we may assume that
$\la{D(J,i)}$ is closed and discrete in $\kw(\K(k-1))$ for every $i\in
I_z\cup I_h$. Since $(\CC,\xi)$ is a sequential scale the set $\cup_{i\in
I_z\cup I_h}D(J,i)$ is independent and the group
$\la{\cup_{i\in I_z\cup I_h}D(J,i)}$ is closed and discrete in
$\kw(\K(k-1))$. If $i\in I_z$ use $D(J,i)\to0$ in $\tau(\U(k-1))$ and
Lemma~\ref{z.group} to find, after thinning out $J$ if necessary, an fpse
$\gks{_{\gamma(k-1)+1}}$ of $\gks{(k-1)}$ such that $D(J,i)\to0$ in
$\kw(\K_{\gamma(k-1)+1})$, $D(J,i')\to0$ in $\tau(\U_{\gamma(k-1)+1})$ for every
$i'\in I_z\setminus\{i\}$ and $D(J,i'')$ is closed and discrete in
$\kw(\K_{\gamma(k-1)+1})$ for every $i''\in I_h\cup
I_z\setminus\{i\}$. Note that while formally Lemma~\ref{z.group} does not
guarantee that each $D(J,i'')$ remains closed and discrete in
$\kw(\K_{\gamma(k-1)+1})$ for $i''\in I_h\cup I_z\setminus\{i\}$, it
follows from $D(J,i'')\to0$ in 
$\tau(\U_{\gamma(k-1)+1})$ that if $D(J,i'')$ is not closed and discrete
in $\kw(\K_{\gamma(k-1)+1})$, after passing to a smaller $J$ if
necessary we may assume that $D(J,i'')\to0$ in
$\kw(\K_{\gamma(k-1)+1})$. We may thus move such $i''$ to $I_u$ and
proceed with the argument. Alternatively, one may note that the fpse
constructed above is of finite type and use Lemma~\ref{w1.m.hd}.

Repeating this argument for every $i\in I_z$ in the natural order and
possibly passing to a subset of $J$ if necessary one can build a
finite fpse-chain
$\set\gks{_\gamma}:\gamma(k-1)<\gamma\leq\gamma_z=\gamma(k-1)+|I_z|.$
above $\gks{(k-1)}$ such that $D(J,i)\to0$ in $\kw(\K_{\gamma_z})$ for
every $i\in I_u\cup I_z$ and $D(J,i)$ is closed and discrete in
$\kw(\K_{\gamma_z})$ for every $i\in I_h$.

Let $i\in I_h$. Use Lemma~\ref{p.stable} to find an fpse-chain
$\set\gks{_\gamma}:\gamma_z<\gamma<\gamma_s(i).$ above
$\gks{_{\gamma_z}}$ close to $\lasi$ and an infinite $J_i\subseteq J$
such that $D(J_i,i)$ is closed, discrete, and stable in
$\gks{_{\gamma_s(i)}}$. Note that $D(J,i')$ remains closed
discrete in $\gks{_{\gamma_s(i)}}$ for every $i'\in I_h$ by
Lemma~\ref{w1.m.hd}.

Repeating this argument for every $i\in I_h$ in the natural order and
possibly passing to a subset of $J$ if necessary one may build an
fpse-chain $\set\gks{_\gamma}:\gamma_z<\gamma<\gamma_s.$ of finite
type over $\gks{_{\gamma_z}}$ such that each $D(J,i)$, $i\in I_h$ is
closed, discrete, and stable in $\gks{_{<\gamma_s}}$. Put
$\delta_i=\hd{\K_{<\gamma_s}}{D(J,i)}{\lasoi}$ and note that
$\delta_i<\sigma^i-1$ for every $i\in I_h$.

For every $i\in I_h$ pick a $K(i)\in\K_{<\gamma_s}$, a $D^+(J,i)=\set
d_j(i):j\in J.\subseteq[\lasoi]_{\delta_i}^{\kw(\K_{<\gamma_s})}$ and
$d^i_j=d_j(i)+a_j(i)$ for every $j\in J$ where $a_j(i)\in K(i)$. By
thinning out $J$ we may assume that $a_j(i)\to a(i)$ for some
$a(i)\in K(i)$. 

Let $\nu=|I_h|$ and $I_h=\{\,i_1,\ldots,i_\nu\,\}$ where the order of
$i_j$ is chosen so that the following assumption holds. Let successor
$\eta^i(k)<\omega_1$ for $i\in I_h$ be chosen so that
$\sigma^i-1\geq\delta_i+\eta^i(k)\geq\sigma^i_k$ and assume that
$\eta^{i_j}(k)$ are increasing for each $k\in\omega$. Put $d_n'=\sum_{i\in I_u\cup
I_z}d_n(i)$. Then $d_n'\to0$ in
$\kw(\K_{<\gamma_s})$. Use Lemma~\ref{wggs.w} to
construct a wggs
$$
W_1=\set(G_\gamma, \K_\gamma, \U_\gamma, S_\gamma,
s_\gamma):\gamma_s\leq\gamma\leq\gamma(i_1,k).
$$
above
$$
(D^+(J,i_1),
G_{<\gamma_s}, \K_{<\gamma_s}, \U_{<\gamma_s}, \delta_{i_1})
$$
so that
$$
s_{\gamma(i_1,k)}\in[\set
d_n(i_1)+d_n':n\in J.]_{\eta^{i_1}(k)}^{\kw(\K_{\gamma(i_1,k)})}
$$
and
$$\so(s_{\gamma(i_1,k)}, \lasoi, \kw(\K_{\gamma(i_1,k)}))=
\delta_{i_1}+\eta^{i_1}(k)\geq\sigma^{i_1}_k.
$$

Recursively applying Lemma~\ref{wggs.w} and using
$\delta_{i_j}+\eta^{i_{j-1}}\leq\delta_{i_j}+\eta^{i_j}$ for $j>1$ construct
wggs
$$
W_j=\set(G_\gamma, \K_\gamma, \U_\gamma, S_\gamma,
s_\gamma):\gamma(i_{j-1},k)<\gamma\leq\gamma(i_j,k).
$$
above
$$
(D^+(J,i_j),
G_{\gamma(i_{j-1},k)}, \K_{\gamma(i_{j-1},k)}, \U_{\gamma(i_{j-1},k)}, \delta_{i_j})
$$
for every $j\leq\nu$ so that
$$
\sum_{l\leq j}s_{\gamma(i_l,k)}\in
[\set\sum_{l\leq j}d_n(i_l)+d_n':n\in
J.]_{\eta^{i_j}(k)}^{\kw(\K_{\gamma(i_j,k)})}
$$
and
$$
\so(s_{\gamma(i_j,k)}, \lasoi, \kw(\K_{\gamma(i_j,k)}))=
\delta_{i_j}+\eta^{i_j}(k)\geq\sigma^{i_j}_k.
$$

Put $\gamma(k)=\gamma(i_\nu,k)$, $\gks{(k)}=\gks{_{\gamma(k)}}$ and
$d(i,k)=a(i)+s_{\gamma(i,k)}$ if $i\in I_h$, otherwise put
$d(i,k)=0$.

If $U\in\U(k-1)$ then $D(J,i)\ain U$ and $\set a_j(i):j\in J.\ain
a(i)+U$ so $D^+(J,i)\ain a(i)+U$. Now 
$d(i,k)\in\cl{a(i)+D^+(J,i)}^{\kw(\K(k))}\subseteq
\cl{U}^{\tau(\U(k))}$.

By the construction, $\sum_{i\in I_h}s_{\gamma(i,k)}\in[\set\sum_{i\in
I}d_n(i):n\in J.]_{\eta^{i_\nu}(k)}^{\kw(\K(k))}$ and $\sum_{i\in
I}a_n(i)\to\sum_{i\in I}a(i)$, so $\sum_{i\in
I}d(i,k)\in[\set\sum_{i\in I}d^i_n:n\in J.]_{\eta(k)}^{\kw(\K(k))}$ by
Lemma~\ref{set.additivity},
where $\eta(k)=\eta^{i_\nu}(k)\leq\max\set\sigma^i-1:i\in I.$.

Since $p_n-\sum_{i\in I}d^i_n\in K$ for all but finitely many $n\in J$
by the choice of $p_{n(j)}$, there exists a
$p(k)\in[P_\alpha]_{\eta(k)}^{\kw(\K(k))}$
such that $p(k)-\sum_{i\in I}d(i,k)\in K$ by Lemma~\ref{cs.sos}. Since
$d_n^i,p_n,p_n-\sum_{i\in I}d^i_n\in\cl{U}^{\tau(\U(k))}$ for all but
finitely many $n\in J$, $p(k)\in\cl{U}^{\tau(\U(k))}$
so~\ref{p.k.gshift} holds.

Let $g\in G_{\gamma(i,k)}\setminus G_{<\gamma(i,k)}$ and
$\so(g, \lasoi, \kw(\K(k)))<\sigma^i_k\leq\sigma^i-1$.
Then $i=i_j\in I_h$ for some $j\leq\nu$ by~\ref{wggs.order},
Lemma~\ref{w1.m.hd} and the
construction of $W_j$.
Then by
Lemma~\ref{fpsestack} $\so(g, \lasoi, \kw(\K_{\gamma(i,k)}))<\sigma^i_k$
contradicting $\so(g, \lasoi, \kw(\K_{\gamma(i,k)}))\geq
\so(s_{\gamma(i,k)}, \lasoi, \kw(\K_{\gamma(i,k)}))\geq\sigma^i_k$
by~\ref{wggs.so.other} and the construction of $W_j$.
Thus~\ref{p.k.so.bound} holds.

Let $\gamma(\omega)=\lim_{k\in\omega}\gamma(k)$ and
$\gks{(\omega)}=\gks{_{<\gamma(\omega)}}$.

If $U'\in\U(\omega)$ then $U'=\cl{U}^{\tau(\U(\omega))}$ where
$U\in\U(k-1)$ for some $k\in\omega$. By~\ref{p.k.gshift} $d(i,n),
p(n), p(n)-\sum_{i\in I}d(i,n)\in\cl{U}^{\tau(\U(n))}\subseteq U'$ for
all $n\geq k$. Since $b(n)=p(n)-\sum_{i\in I}d(i,n)\in K$
by~\ref{p.k.gshift}, it follows that $b(n)\to0$ in
$\kw(\K(\omega))$.

By thinning out and reindexing we may assume that $d(i,n)\not=0$ if
$i\in I_0$ and $d(i,n)=0$ if $i\not\in I_0$ for some $I_0\subseteq
I$ and every $n\in\omega$. If $g\in G_{\gamma(i,n)}\setminus
G_{<\gamma(i,n)}$ then
$\so(g, \lasoi, \kw(\K_{\gamma(i,n)}))\geq\sigma^i_n$
by~\ref{p.k.so.bound} and $\gamma(i,n)$ is cofinal in
$\lim_{k\in\omega}\gamma(k)$. Thus by~\ref{p.k.d} and Lemma~\ref{wggs.c}
$\hd{\K(\omega)}{D^i}{\lasoi}=\sigma^i-1$ where $D^i=\set
d(i,n):n\in\omega.$ for every $i\in I_0$ and $\la{\cup_{i\in I_0}D^i}$
is closed and discrete in $\kw(\K(\omega))$.

Using Lemma~\ref{z.group} we can build a
finite fpse-chain
$$
\set\gks{_\gamma}:
\gamma(\omega)<\gamma\leq
\gamma(\omega)+s=\gamma_t.
$$ 
above $\gks{(\omega)}$ relative to $(\CC,\xi)$ such that $D^i_0\to0$
in $\kw(\K_{\gamma_t})$ for each $i\in I_0$ where $D^i_0=\set
d(i,n):n\in J_0.$ for some infinite $J_0\subseteq\omega$. Then $\set
p(n):n\in J_0.\to0$ and
$p(n)\in[P_\alpha]_{\eta'}^{\kw(\K_{\gamma_t})}$ where
$\eta'=\sup_{k\in\omega}\eta(k)\leq\max\set\sigma^i-1:i\in I.$ so
$0\in[P_\alpha]_{\eta}^{\kw(\K_{\gamma_t})}$ where
$\eta\leq\max\set\sigma^i:i\in I.$. Put $\mcofl(\alpha+2)=\gamma_t$.

Put $\mcofl(\alpha+3)=\mcofl(\alpha+2)+1$. Let $\gkup$ be an arbitrary
fpse of $\gks{_{\mcofl(\alpha+2)}}$. Using
Lemmas~\ref{ds.closed.sums}, \ref{ct.resolve}, and \ref{ssc.resolve}
extend $\U'$ to a countable family of open (in
$\kw(\K_{\mcofl(\alpha+2)})$) subgroups of $G_{\mcofl(\alpha+2)}$ of
finite index $\U''\supseteq\U'$
so that~\ref{g.resolve} holds after replacing $\U_{\mcofl(\alpha+3)}$
with $\U''$ and let
$\gks{_{\mcofl(\alpha+3)}}=(G', \K', \U'')$.
\end{proof}

\def\Kw{\K_{\omega_1}}
\def\Uw{\U_{\omega_1}}
\def\Gw{G_{\omega_1}}

Suppose $\diamondsuit$ holds and let $\set C_\alpha:\alpha<\omega_1.$ be a
$\diamondsuit$-sequence. Identifying $\omega_1$ and $2^\omega$ we may assume that
each $C_\alpha\subseteq2^\omega$.
Let $\set P_\alpha:\alpha<\omega_1\hbox{ is a
limit ordinal}.$ list all infinite countable subsets of
$2^\omega$ so that each $P_\alpha$ is listed $\omega_1$ times.

\begin{lemma}\label{e}
Let $\gks{_{\omega_1}}=\gks{_{<\omega_1}}$ where $\gks{_\alpha}$,
$\alpha<\omega_1$ have been constructed in Lemma~\ref{g.ind}. Then $\kw(\Kw)=\tau(\Uw)$,
$\Gw$ is countably compact, and
$\so(\kw(\Kw))=\sup\set\sigma^m:m\in\omega.$.
\end{lemma}
\begin{proof}
Suppose $A\subseteq\Gw$ is such that $0\not\in A$, $A\cap K$ is closed
for every $K\in\Kw$ and $A\cap U\not=\varnothing$ for every
$U\in\Uw$.

Let $\theta$ be a large enough cardinal. Consider the sets of the form $A\cap
M$ where $M$ is a countable elementary submodel of $H(\theta)$
and $X\in M$ is a countable set containing the details of the construction
of $\Gw$. The set
$$
\set\gamma\in\omega_1:\gamma=M\cap\omega_1, X\in M, M\leq H(\theta).
$$
is a club in $\omega_1$. Thus $C_\gamma=A\cap M$ for some
$\gamma<\omega_1$ where $M\cap\omega_1=\gamma$. Note that
$\gamma=\mcofl(\alpha)$ for some limit $\alpha<\omega_1$ and
$\ccl{C_\alpha}=A\cap G_{<\mcofl(\alpha)}$.

Let $\K_{<\gamma}=\set K_i:i\in\omega.$, $\U_{\omega_1}\cap M=\set
U_i:i\in\omega.$. Note that for any $U\in\U_{<\gamma}$ there is a
$U_i$ such that $U_i\cap G_{<\omega}=U$ (since
$\cl{U}^{\K_{\omega_1}}\in\U_{\omega_1}$). Pick points $d_n\in A\cap (\cap_{i\leq n}U_i)\cap M$
so that
$$
d_n\not\in\sum_{i<n}K_i+\la{\set d_i:i<n.}+\sum_{m<n}\cl{\lasm}^{\kw(\K_{<\gamma})}
$$
If the recursion does not terminate by~\ref{g.lus.c} there exists an
$S\subseteq D=\set d_n:n\in\omega.\subseteq A$ such that $S\to0$ in
$\kw(\K_{\mcofl(\alpha+1)})$ contradicting the choice of $A$.

If the recursion stops at some $n\in\omega$ let $K'\in\K_{<\gamma}\cap
M$ be such that $\sum_{i<n}K_i+\la{\set d_i:i<n.}\subseteq K'$. Then
for every $d\in A\cap(\cap_{i\leq n}U_i)\cap M$ there is an
$s\in\sum_{m<n}\overline{\lasm}^{\kw(\K_{<\gamma})}\subseteq\sum_{m<n}\cl{\lasm}^{\kw(\Kw)}$
such that $d\in s+K'$.

By elementarity $A\cap U\subseteq\sum_{i\leq
m}\cl{\lasi}^{\kw(\Kw)}+K$ for some $m\in\omega$, $U\in\U_{<\omega_1}$,
and $K\in\Kw$. We may assume that $A=A\cap U$ by picking a subset if
necessary. 

Pick points $d_n\in A\cap (\cap_{i\leq n}U_i)\cap M$ so that
$d_n=a_n+\sum_{i\leq m}d_n^i$ where
$d_n^i\in\cl{\lasi}^{\kw(\K_{<\gamma})}$ and $a_n\in K$.
By passing to a subsequence and reindexing if necessary, assume
$a_n\to a$. Now $d_n\to0$ so $\sum_{i\leq m}d_n^i\to a$ in
$\tau(\U_{<\gamma})$. Let $K\in\K_\beta$ for some $\beta<\gamma$.

If $a\not\in\sum_{i\leq m}\cl{\lasi}^{\kw(\K_{\beta'})}$ for any
$\beta'<\gamma$ let $\alpha'\in\omega_1$ be a limit such that
$\beta<\mcofl(\alpha'+3)<\gamma$. Such $\alpha'$ exists since
$\beta, \mcofl\in M$. Then $a\not\in\sum_{i\leq
m}\cl{\lasi}^{\kw(\K_{\mcofl(\alpha'+3)})}$ and by~\ref{g.resolve}
there is a clopen $U\in\U_{\mcofl(\alpha'+3)}$ such that $a\not\in
U$ and $\sum_{i\leq m}{\lasi}\subseteq U$. Note that $U=U_i\cap
G_{\mcofl(\alpha'+3)}$ for some $i\in\omega$ so $a\not\in U_i$ and $\sum_{i\leq
m}d_n^i\in U_i$ for every $n\in\omega$ contradicting $\sum_{i\leq m}d_n^i\to a$ in $\tau(\U_{<\gamma})$.

Thus
$a\in\sum_{i\leq m}\cl{\lasi}^{\kw(\K_{<\gamma})}$. By~\ref{g.cc} there
exists a $\gamma_0<\omega_1$ such that $\gamma_0>\gamma$ and $D^i\to d^i$ in
$\kw(\K_{\gamma_0})$ for each $i\leq m$ where $D^i\subseteq\set
d_n^i:n\in\omega.$ is infinite. By thinning out and reindexing, we may
assume that $d_n^i\to d^i$ in $\kw(\K_{\gamma_0})$. Then $d_n\to
d=a+\sum_{i\leq m}d^i$ in $\kw(\K_{\gamma_0})$ so $d\in
A\cap\sum_{i\leq m}\cl{\lasi}^{\kw(\K_{\gamma_0})}$ by
Lemma~\ref{set.additivity}.

Let $U\in\Uw\cap M$ then $d_n\in U$ for all but
finitely many $n\in\omega$ so $d\in U$. Thus
for every $U\in\Uw\cap M$ there is a $\gamma'<\omega_1$ and a $d\in
A\cap(\sum_{i\leq
m}\cl{\lasi}^{\kw(\K_{\gamma'})})\cap\cl{U}^{\tau(\U_{\gamma'})}$. Then
the set $A'=A\cap(\sum_{i\leq m}\cl{\lasi}^{\kw(\Kw)})$ is closed in
$\kw(\Kw)$ and by elementarity $0\in\cl{A'}^{\tau(\Uw)}$ so assume
that $A\subseteq\sum_{i\leq m}{\lasi}^{\kw(\Kw)}$.

Let $I\subseteq m$ be such that $0\in\cl{A'}^{\tau(\Uw)}$ and
$0\not\in\cl{A''}^{\tau(\Uw)}$ for any $A''=A\cap(\sum_{i\in
I'}\cl{\lasi}^{\kw(\Kw)})$ where $A'=A\cap(\sum_{i\in
I}\cl{\lasi}^{\kw(\Kw)})$, $I'\subset I$, $I'\not=I$. Using this
property of $I$ and picking a closed (in $\kw(\Kw)$)
subset of $A'$ if necessary we may assume that
every $d\in A$ has the property that $d=\sum_{i\in I}d^i\in A$ where $d^i\not=0$ and
$d^i\in\cl{\lasi}^{\kw(\Kw)}$.

Pick points $d_n=\sum_{i\in I}d_n^i\in A\cap(\cap_{k\leq n}U_k)\cap M$ so
that $d_n^i\in\cl{\lasi}^{\kw(\K_{<\gamma})}\cap(\cap_{k\leq n}U_k)$
for every $i\in I$.

If the recursion terminates at some $n\in\omega$ then there exists a
$U\in \Uw\cap M$ such that for every $d\in A\cap M$ $d\not\in\sum_{i\in
I}(\cl{\lasi}^{\kw(\K_{\gamma'})}\cap U)$ for
any $\gamma'\in M$. Let $\gamma'<\gamma$ be such that $U\cap
G_{\gamma'}\in\U_{\gamma'}$. Then $U\cap G_{\gamma''}\in\U_{\gamma''}$
for any $\gamma''\geq\gamma'$. Let $\alpha'<\gamma$ be a limit such
that $\gamma'<\mcofl(\alpha'+3)<\gamma$. By~\ref{g.resolve} there
exists a $U'\in\Uw\cap M$ such that $U'\subseteq U$, $U'\cap
G_{\mcofl(\alpha'+3)}\in\U_{\mcofl(\alpha'+3)}$, and $U'\cap\sum_{i\in
I}\cl{\lasi}^{\kw(\K_{\mcofl(\alpha'+3)})}=\sum_{i\in
I}(\cl{\lasi}^{\kw(\K_{\mcofl(\alpha'+3)})}\cap U')$. Let
$\gamma_0\geq\mcofl(\alpha'+3)$ by such that $\gamma_0<\gamma$ and
$d=\sum_{i\in I}d^i\in A\cap U'\cap M$ for some
$d^i\in\cl{\lasi}^{\kw(\K_{\gamma_0})}$. Such $d$ and $\gamma_0$ exist
by elementarity since $A,U',\Kw,\CC\in M$. Note that for every $i\in I$
the set $\cl{\lasi}^{\kw(\K_{\gamma_0})}\cap U'$ is closed in
$\kw(\K_{\gamma_0})$ so by Lemma~\ref{ds.closed.sums} the set
$\sum_{i\in I}(\cl{\lasi}^{\kw(\K_{\gamma_0})}\cap U')$ is closed in
$\kw(\K_{\gamma_0})$. Since $U'\cap\sum_{i\in I}\lasi$ is dense in
$U'\cap\sum_{i\in I}\cl{\lasi}^{\kw(\K_{\gamma_0})}$ this shows that
$U'\cap\sum_{i\in I}\cl{\lasi}^{\kw(\K_{\gamma_0})}=\sum_{i\in
I}(\cl{\lasi}^{\kw(\K_{\gamma_0})}\cap U')$. Then
$d^i\in\cl{\lasi}^{\kw(\K_{\gamma_0})}\cap
U'\subseteq\cl{\lasi}^{\kw(\K_{\gamma_0})}\cap U$ (note that $d$ 
uniquely determines $d^i$ by~\ref{seq.scale.direct}) contradicting the
choice of $U$.

Thus $D^i=\set d_n^i:n\in\omega.\to0$ in $\tau(\U_{<\gamma})$ for each
$i\in I$. If $D^i$ is not closed and discrete in $\kw(\K_{<\gamma})$
for some $i\in I$ then for some infinite $J\subseteq\omega$ $\set
d_n^i:i\in J.\to0$ in $\kw(\Kw)$.
Using~\ref{g.cc} and possibly thinning out $J$ we may assume
that $\set d_n^j:n\in J.\to
d(j)$ for every $j\in I$ in $\kw(\Kw)$. Then $d=\sum_{i'\in I}d(i')\in\sum_{i'\in
I\setminus\{i\}}\cl{\lasi}^{\kw(\Kw)}$ and $d\in A$ contradicting the
minimality of $I$.

Thus each $D^i$ is closed and discrete in
$\kw(\K_{<\gamma})$. If $\hd{\K_{<\gamma}}{D^{i'}}{\la{S^{i'}}}\geq\sigma^{i'}-1$
for some $i'\in I$ then by Lemma~\ref{z.group} there exists an fpse
$\gks{_\gamma}$ of $\gks{_{<\gamma}}$ such that $\set d_n^{i'}:n\in J.\to0$ in
$\kw(\K_\gamma)$ and a finite fpse-chain
$\set\gks{_\lambda}:\gamma<\lambda<\gamma+k.$ for some $k\in\omega$
such that $\set d_n^i:n\in J.\to d(i)$ in $\kw(\K_{<\gamma+k})$ for some $d(i)\in
G_{<\gamma+k}$ and some infinite $J\subseteq\omega$. As above $d=\sum_{i\in I}d(i)\in\sum_{i\in
I\setminus\{i'\}}\cl{\lasi}^{\kw(\K_\gamma)}$ and $d\in\cl{A}^{\K_\gamma}$. Thus by~\ref{g.lus.wggs} $A\cap\sum_{i\in
I'}\cl{\lasi}^{\kw(\K_{\mcofl(\alpha+1)})}\not=\varnothing$ for some
$I'\subset I$, $I'\not=I$ contradicting the minimality of $I$.

Now $\hd{\K_{<\gamma}}{D^i}{\la{S^i}}<\sigma^i-1$ for every $i\in I$. 
Pick an $i'\in I$ and use Lemma~\ref{ct.split} to build an fpse-chain
$\set\gks{_\gamma}:\mcofl(\alpha)\leq\gamma\leq \mcofl(\alpha)+k.$ of finite
type above $\gks{_{<\mcofl(\alpha)}}$ away from $\la{S^{i'}}$ for some
$k\in\omega$ so that for some infinite $J\subseteq\omega$ $\set
d_n^i:n\in J.\to d^i\not\in G_{<\mcofl(\alpha)}$ in $\kw(\K_{\mcofl(\alpha)+k})$ for $i\in
I\setminus\{i'\}$. Then $P=\set d_n^{i'}:n\in J.$ is closed and
discrete in $\kw(\K_{\mcofl(\alpha)+k})$ by Lemma~\ref{w1.m.hd}.

Use Lemma~\ref{stable} to construct an fpse-chain
$\set\gks{_\gamma}:\mcofl(\alpha)+k<\gamma\leq\gamma_0.$ such that, after
possibly thinning out $J$, $P$ is stable in $\gks{_{\gamma_0}}$ and
$\la{P}$ is closed and discrete in $\kw{\K_{\gamma_0}}$.

Let $J=\cup_{j\in\omega}J_j$ be such that $J_j$ are infinite and
disjoint and $P^j=\set d_n^{i'}:n\in J_j.$.

Repeatedly apply Lemma~\ref{wggs.w} to build wggs
$$
W_j=\set(G_\gamma, \K_\gamma, \U_\gamma, S_\gamma,
s_\gamma):\gamma_j<\gamma\leq\gamma_{j+1}.
$$
over $(D_j,
G_{\gamma_j}, \K_{\gamma_j}, \U_{\gamma_j}, \delta)$ 
where
$\delta=\hd{\K_{\gamma_j}}{P^j}{\lo{S^{i'}}}=\hd{\K_{\gamma_0}}{P}{\lo{S^{i'}}}$
and
$D_j\subseteq[\lasoi]_{\delta}^{\kw(\K_{<\gamma_j})}$ is such that
$P^j\subseteq D_j+K(j)$ and $D_j\subseteq P^j+K(j)$ for some
$K(j)\in\K_{\gamma_j}$ (due to the 
stability of $P$ we may assume that $K(j)=K\in\K_{\gamma_0}$ but this
stronger property is not needed) such that
$$
\so(s_{\gamma_{j+1}}, \lo{S^{i'}}, \kw(\K_{\gamma_j}))\geq\sigma^{i'}_{j+1}
$$
and $\la{\cup_{k>j}P^j}$ is closed and discrete in
$\kw(\K_{\gamma_{j+1}})$. Let
$\gamma^0=\lim\gamma_j$.

Use
Corollary~\ref{cs.mseq} to find a
$d(i',n)\in\cl{P^n}^{\kw(\K_{\gamma_{n+1}})}$ such that
$d(i',n)=a(n)+s_{\gamma_{n+1}}$ for some $a(n)\in K(n)$. Note that $d(i',n)\in
G_{\gamma_{n+1}}\setminus G_{<\gamma_{n+1}}$ since $a(n)\in
K(n)\in\K_{<\gamma_{n+1}}$ and by~\ref{wggs.so.other}
$$
\so(g, \lo{S^{i'}}, \kw(\K_{\gamma_{n+1}}))\geq
\so(s_{\gamma_{j+1}}, \lo{S^{i'}}, \kw(\K_{\gamma_{n+1}}))\geq\sigma^{i'}_n
$$
for any $g\in G_{\gamma_{n+1}}\setminus G_{<\gamma_{n+1}}$. Let $d(i,n)=d^i$ for $i\in
I\setminus\{i'\}$. Put $d(n)=\sum_{i\in I}d(i,n)$. Then
$d(n)\in\cl{A}^{\kw(\K_{\gamma^0})}$ by Corollary~\ref{cs.mseq}. Since $D^i\to0$ in
$\tau(\U_{<\gamma})$,
$d(i,n)\in\cap\set\cl{U}^{\tau(\U_{<\gamma^0})}:U\in\U_{<\gamma}.$.

Now by~\ref{g.lus.so.raise} for every $n\in\omega$ there exists a
$d(n)\in G_{\mcofl(\alpha+1)}\setminus G_{<\gamma}$ such that
$d(n)\in\cl{A}^{\kw(\K_{\mcofl(\alpha+1)})}$
and $d(n)=\sum_{i\in I}d(i, n)$ where
$$
d(i, n)\in\cl{\lasi}^{\kw(\K_{\mcofl(\alpha+1)})}\cap
\bigcap\set\cl{U}^{\tau(\U_{\mcofl(\alpha+1)})}:U\in\U_{<\mcofl(\alpha)}.
$$
and
$\so(d(i', n), \lasoi, \kw(\K_{\mcofl(\alpha+1)}))\geq\sigma^{i'}_n$
for some $i'\in I$. Note that $d(i, n)\not\in M$.

Thus for any $U\in\Uw\cap
M$, any $\gamma'\in\omega_1\cap M$, and any $n\in\omega$ there exists a
$d(n)=\sum_{i\in I}d(i,n)\in A$ where
$d(i, n)\in\cl{\lasi}^{\kw(\K_{\mcofl(\alpha+1)})}\cap U$, $d(i, n)\not\in
G_{\gamma'}$, and
$\so(d(i', n), \lasoi, \kw(\K_{\mcofl(\alpha+1)}))\geq\sigma^{i'}_n$
for some $i'\in I$, $\alpha<\omega_1$.

By elementarity such $d(n)$ and $d(i, n)$ exist in every
$(\cap_{j\leq n}U_j\setminus G_{<\gamma_n})\cap M$
for some $\gamma_n$ cofinal in $\gamma$. Now
$\hd{\K_{<\gamma}}{\set
d({i'}, n):n\in\omega.}{\lasoi}\geq\sigma^{i'}-1$ by
Lemma~\ref{wggs.c} and $d(i', n)\to0$ in $\U_{<\mcofl(\alpha)}$
so by an argument similar to the one above using~\ref{g.lus.wggs} one obtains a contradiction
with the minimality of $I$.

The upper bound estimate for $\so(\Gw)$ follows from \ref{g.so}, while
the lower bound estimate follows from Lemma~\ref{fpsestack}. Countable
compactness follows from~\ref{g.cc}.
\end{proof}
\begin{theorem}[$\diamondsuit$]\label{cc.g}
Let $\sigma\leq\omega_1$. There exists a countably compact sequential
group $G$ such that $\so(G)=\sigma$.
\end{theorem}
\begin{proof}
The existence of such $G$ for $\sigma<\omega_1$ follows from
Lemma~\ref{e}. For $\sigma=\omega_1$ consider a $\Sigma$ product of
$\omega_1$ countably compact sequential groups $G_\alpha$ such that
$\so(G_\alpha)=\alpha$.

Note that the group of sequential order $\omega_1$ constructed above
is not separable unlike the examples of smaller sequential order. A
modified construction similar to that of
Lemma~\ref{e} may be used to construct a separable example of
sequential order $\omega_1$, as well, although we omit the details. 
\end{proof}
\begin{corollary}
The existence of a countably compact sequential non Fr\'echet group is
independent of the axioms of ZFC.
\end{corollary}
\begin{proof}
See Theorem~\ref{cc.g} and \cite{Shi1} Theorem~2.
\end{proof}

We conclude by listing some open questions.

\begin{question}
Does Theorem~\ref{cc.g} follow from
CH alone?
\end{question}

\begin{question}
Is the existence of a countably compact sequential group $G$ such that
$K\subseteq G$ for some compact subspace $K$, $\so(K)\geq2$
consistent with ZFC?
\end{question}

\begin{question}
Can the construction in Theorem~\ref{cc.g} be made
Cohen-indestructible (see, for example,~\cite{Hr1} for the relevant
definitions)?
\end{question}

The last question requires some clarification. Since the group $G$
will contain compact subspaces homeomorphic to the Cantor cube
$2^\omega$, the addition of {\em any\/} reals will destroy
the sequentiality of $G$. Each such subspace, as well as $G$ itself, will inherit
a precompact topology 
from the ground model, however, so the group $\cl{G}\supseteq G$ in the extension may
be defined by taking the Raikov completion of each compact subspace in
$\Kw$ and letting the topology of $\cl{G}$ be determined by the new family
of compact subspaces.

\section{Acknowledgements.}
The research in this paper was started when one of the authors
(Alexander Shibakov) was visiting the Department of Mathematics at
Ehime University in the Summer of 2017. He would like to thank the
Department and Prof.\ Shakhmatov for the support that made this visit
possible
and productive.

\end{document}